\documentclass[10pt]{amsart}
\usepackage{amsmath,amssymb,amsthm,array,hyperref,color,cite,graphicx,mathrsfs}
\usepackage[all]{xypic}

\newtheorem{thm}{Theorem}[section]
\newtheorem{cor}[thm]{Corollary}
\newtheorem{lem}[thm]{Lemma}
\newtheorem{prop}[thm]{Proposition}
\theoremstyle{definition}

\newtheorem{df}[thm]{Definition}

\newtheorem{question}[thm]{Question}
\newtheorem{rmk}[thm]{Remark}

\newcommand{\R}{\mathbb R}

\newcommand{\Z}{\mathbb Z}

\newcommand{\Sph}{\mathbb S}

\newcommand{\Map}{\textup{Map}}

\newcommand{\id}{\textup{id}}
\newcommand{\colim}{\textup{colim}\,}
\newcommand{\hocolim}{\textup{hocolim}\,}
\newcommand{\holim}{\textup{holim}\,}

\newcommand{\ra}{\longrightarrow}
\newcommand{\la}{\longleftarrow}
\newcommand{\sma}{\wedge}
\newcommand{\barsmash}{\,\overline\wedge\,}
\newcommand{\ti}{\widetilde}
\newcommand{\simar}{\overset\sim\longrightarrow}

\newcommand{\mc}{\mathcal}
\newcommand{\op}{\textup{op}}

\newcommand{\sk}{\textup{Sk}}

\newcommand{\uda}{\rotatebox[origin=c]{180}{A}}

\title{Coassembly and the $K$-theory of finite groups}
\author{Cary Malkiewich}
\keywords{algebraic $K$-theory, assembly, transfer map, parametrized spectrum}

\begin{document}

\maketitle
\begin{abstract}
We study the $K$-theory and Swan theory of the group ring $R[G]$, when $G$ is a finite group and $R$ is any ring or ring spectrum. In this setting, the well-known assembly map for $K(R[G])$ has a companion called the coassembly map. We prove that their composite is the equivariant norm of $K(R)$. This gives a splitting of both assembly and coassembly after $K(n)$-localization, a new map between Whitehead torsion and Tate cohomology, and a partial computation of $K$-theory of representations in the category of spectra.
\end{abstract}

\parskip 0ex
\tableofcontents
\parskip 2ex

\section{Introduction.}

Algebraic $K$-theory provides a deep set of invariants for each ring $R$, in the form of a sequence of abelian groups $K_i(R)$. In topology, these groups provide obstructions for classical problems such as recognizing finite CW complexes, classifying finite group actions on spheres, and trivializing smooth cobordisms.

In several of these applications, the most important computation is how the $K$-theory of a group ring $R[G]$ is related to the $K$-theory of $R$. They are connected by the \emph{assembly map}:
\begin{equation}\label{assembly_groups}
H_*(BG;K(R)) \ra K_*(R[G])
\end{equation}
We think of the left-hand side of (\ref{assembly_groups}) as very computable when compared to the right-hand side.
So we may construct classes in $K_i(R[G])$, by building them first in $H_i(BG;K(R))$.
However it is difficult to tell whether the classes built this way are actually nonzero.
We therefore ask
\begin{question}
Is the assembly map injective, or even an isomorphism?
\end{question}

This question has been studied extensively in many contexts.
The \emph{$K$-theoretic Novikov conjecture} states that (\ref{assembly_groups}) is rationally split injective when $R = \Z$.
This was proven by Bokstedt, Hsiang, and Madsen for any group $G$ whose homology has finite type \cite{bhm}.
The \emph{Farrell-Jones conjecture} (in one instance) states that a nonconnective version of (\ref{assembly_groups}) is an isomorphism when $R$ is regular and $G$ is torsion-free. Farrell and Jones proved this isomorphism rationally for $R = \Z$ and $G$ a discrete cocompact subgroup of a Lie group with finitely many path components \cite{farrell1993isomorphism}. There also are variants of these conjectures for $L$-theory, implying the Novikov conjecture and Borel conjecture, respectively.

The integral version of the Farrell-Jones conjecture has been quite difficult.
Many results place additional restrictions on $R$ and $G$, and even then only get injectivity, which is sometimes called the \emph{Integral Novikov conjecture}.
Carlsson and Pedersen proved injectivity for a large class of groups $G$ with finite $BG$, including the word hyperbolic groups \cite{carlsson1995controlled}.
Bartels and Reich proved integral isomorphism for any ring $R$, when $G$ is the fundamental group of a Riemannian manifold of negative sectional curvature \cite{bartels2005farrell}.
There are quite a few more results by several different authors, and our brief summary does not do them justice.
A comprehensive survey can be found in \cite{luckbaum}; more recent results include \cite{bartels2012borel} and \cite{wegner2015farrell}.

In light of this earlier work, we know the most about assembly when $G$ is an infinite discrete group, and $BG$ is finite in some sense.
On the other hand, we know very little about the case of $G$ finite, other than injectivity in low degrees \cite{luckbaum}. In this case, we present a theorem that splits the assembly map after a certain localization.

To describe our result, we first re-interpret assembly as a map of spectra
\[ BG_+ \sma K(R) \ra K(R[G]). \]
We recall that $K(R[G])$ has a close cousin called Swan theory, or simply $G$-theory, $G^R(R[G])$ \cite{evans1986k}.
The Swan theory studies modules over $R[G]$ which are perfect (finitely generated projective) over $R$, instead of those that are perfect over $R[G]$. In other words, it is the $K$-theory of the \emph{representations} of $G$ in the category of $R$-modules. The assembly map is defined by a universal property, as we recall in section \ref{assembly}, and the dual of this construction gives a \emph{coassembly map}
\[ G^R(R[G]) \ra F(BG_+,K(R)) \]
out of Swan theory. When $G$ is finite there is also a \emph{Cartan map}
\[ K(R[G]) \ra G^R(R[G]) \]
which takes each perfect $R[G]$-module $M$ to itself, regarded as an $R[G]$-module whose underlying $R$-module is perfect.

Of course, we can define $K$-theory and Swan theory for each ring spectrum $R$ and topological group $G$.
When $R$ is the sphere spectrum, this yields Waldhausen's $A$-theory $A(BG)$, and a contravariant version called $\uda(BG)$.
This kind of $K$-theory is well-studied, but the corresponding Swan theory is poorly understood.
The Swan theory of $R[G]$ studies the representations of $G$ in the category of $R$-module spectra, so it might be called \emph{spectral representation theory.}

Now we can state the main theorem.
\begin{thm}\label{assembly_coassembly_norm}
If $G$ is a finite group and $R$ is a ring or ring spectrum, then the composite
\[ \xymatrix @C=0.6in{ BG_+ \sma K(R) \ar[r]^-{\textup{assembly}} & K(R[G]) \ar[r]^-{\textup{Cartan}} & G^R(R[G]) \ar[r]^-{\textup{coassembly}} & F(BG_+,K(R)) } \]
is homotopic to the norm map on $K(R)$ with the trivial $G$-action,
\[ K(R)_{hG} \ra K(R)^{hG}. \]
\end{thm}

One can give a rough but intuitive sketch of the argument as follows. If we ignore the $BG$ terms, then the assembly map is given by the $K$-theory of the map of categories that sends the $R$-module $M$ to the perfect $R[G]$-module $G \otimes M = \bigoplus^G M$. The Cartan and coassembly maps do not change the module, but merely forget some of the structure. We are left with the self-map of the category of perfect $R$-modules sending $M$ to $\bigoplus^G M$. By the additivity property of $K$-theory, this induces a map on $K(R)$ which is a $G$-fold sum of the identity map with itself. As we rove around $BG$ we get a family of such maps, but the ordering of the $G$-fold sum is shuffled, so the resulting map of spectra is a transfer. The transfer corresponding to the norm map is a $G$-fold sum with $G \times G^\op$-monodromy corresponding to the left and right multiplication actions of $G$ on itself (see $\S 7$). So it remains to show that the $G$-actions coming from the assembly and coassembly maps have exactly this behavior. This is carried out in $\S 5,6$ for categories of $G$-spaces, and lifted to the relevant categories of $G$-spectra using technical constructions developed in $\S 4$.


One corollary of our result is a $K(n)$-local Novikov conjecture for finite groups:
\begin{cor}\label{cor_first}
Let $G$ be a finite group and $R$ any ring or ring spectrum. Then the assembly map
\[ BG_+ \sma K(R) \overset\alpha\ra K(R[G]) \]
is split injective after $K(n)$-localization, at any prime $p$ and for any $n \geq 0$.
\end{cor}

This is discussed further in $\S$\ref{applications}. We also remark there that Theorem \ref{assembly_coassembly_norm} implies the cofiber of assembly admits a map to the cofiber of the norm. This leads to a map from Whitehead groups to Tate cohomology that we believe warrants further study.
\begin{cor}\label{cor_second}
If $G$ is a finite group, there is a map from its Whitehead group to its Tate cohomology with coefficients in $A(*)$ or $K(\Z)$:
\[ \xymatrix{ Wh(G) \ar[r] & \pi_1(A(*)^{tG}) \ar[r] & \pi_1(K(\Z)^{tG}) } \]
\end{cor}

We will also prove a variant of Theorem \ref{assembly_coassembly_norm} which, when combined with the Segal conjecture, allows us to partially compute some rings of representations in spectra:
\begin{thm}\label{intro_burnside_split}
If $R$ is a ring spectrum augmented over the sphere, and $G$ is a finite $p$-group, then the group $\pi_0(G^R(R[G])^\wedge_p)$ contains the Burnside ring $A(G)^\wedge_p$ as a direct summand.
\end{thm}

This suggests that the analog of Artin's induction theorem, relating the representation ring of $G$ to that of the cyclic subgroups of $G$, will fail for representations of $G$ in spectra.
In the language of \cite{mathew2015derived}, we should expect such a theory to have a larger derived defect base.

Finally we summarize the technical results that may be of independent interest. It is well-known that one can build Waldhausen's $A$-theory out of a category of parametrized spectra, instead of spaces, because $K$-theory is preserved under stabilization. Indeed, such a category of spectra is necessary to define the assembly and coassembly maps by a universal property.

However, it seems that there does not yet exist a category of parametrized spectra that is robust enough to allow for a definition of $A(B)$, $\uda(B)$, and the Cartan map between them. The model category of \cite{ms} does not work, since a pullback of a cofibration is not a cofibration.

We remedy the situation by constructing some Waldhausen categories of parametrized spectra which are both geometrically and homotopically well-behaved, and construct external pairings on these categories similar to those found in \cite{weiss2000products,williams2000bivariant}:

\begin{prop}
For each ring spectrum $R$ there is a covariant homotopy functor $A(B;R)$ from unbased spaces to spectra.
When $G$ is a topological group, $A(BG;R)$ is equivalent to $K(R[G])$. Given two rings $R$ and $S$ there is a pairing
\[ A(B;R) \sma A(B';S) \ra A(B \times B';R \sma S) \]
which has the obvious associativity property.
\end{prop}

\begin{prop}
For each ring spectrum $R$ there is a contravariant homotopy functor $\uda(B;R)$ from unbased spaces to spectra.
When $G$ is a topological group, $\uda(BG;R)$ is equivalent to $G^R(R[G])$. There is a pairing as above, and if $G$ is a finitely dominated topological group, there is a \emph{Cartan map}
\[ A(BG;R) \ra \uda(BG;R) \]
If $G$ and $R$ are discrete, this map gives on $\pi_0$ the Cartan map of \cite{evans1986k}.
\end{prop}

This lets us interpret $K(R[G])$ and $G^R(R[G])$ as functors on unbased spaces, so that we may define the assembly and coassembly maps
\[ \xymatrix @R=0.4em { BG_+ \sma K(R) \ar[r] & A(BG;R) \\
\uda(BG;R) \ar[r] & \Map_*(BG_+,K(R)) } \]
as universal approximations by linear functors.
We give an explicit proof that this assembly map agrees with the classic $K$-theory assembly map, and give a similar combinatorial formula for coassembly.
Our proof of Thm \ref{assembly_coassembly_norm} uses the combinatorial formulas for these maps, but the universal property will be important for connecting our theorem to applications.

The outline of the paper is as follows. In $\S$2 we discuss in detail the corollaries of our main theorem. In $\S$3 we begin the technical work, reviewing Waldhausen $K$-theory and giving a modern adaptation of Waldhausen's observation that the spectrum $K(C)$ can be constructed by a classical delooping machine. In $\S$4 we construct the Waldhausen categories of spaces and spectra that give the functors $A(B;R)$ and $\uda(B;R)$. In $\S$5 we review the universal properties and constructions of assembly and coassembly. In $\S$6 we modify these maps by a homotopy to explicit simplicial maps, which lift to the $K$-theory of finite sets. In $\S$7 we recognize the composite of these maps as the norm, using the $E_\infty$ structure on $\Omega^\infty K(C)$ from $\S 3$. This proves Theorem \ref{assembly_coassembly_norm}. In $\S$8, we prove a generalization of Theorem \ref{assembly_coassembly_norm} to all subgroups and conclude Theorem \ref{intro_burnside_split}.

The author is grateful to acknowledge Mark Behrens, Andrew Blumberg, Dustin Clausen, Ralph Cohen, John Greenlees, Jesper Grodal, John Klein, Akhil Mathew, Randy McCarthy, Mona Merling, Justin Noel, and Bruce Williams for their help with and enlightening conversations about this project. He thanks Mark Ullmann and Xiaolei Wu for helpful comments on the exposition of the Farrell-Jones conjecture, and the anonymous referee for helpful comments on the exposition throughout the paper. The $K$-theoretic results in this paper are motivated by $THH$-theoretic results in the author's thesis, which was written under the direction of Ralph Cohen at Stanford University.

\section{Applications and open questions.}\label{applications}

In this section we present a bit more background on the applications of Theorem \ref{assembly_coassembly_norm}.

Here is one application. Recall that the \emph{Whitehead group} $Wh(G)$ of a discrete group $G$ is the cokernel of the inclusion
\[ G^{ab} \oplus \Z/2 \cong H_1(BG) \oplus K_1(\Z) \cong \pi_1(BG_+ \sma K(Z)) \ra K_1(\Z[G]) \]
This may be identified with $\pi_0$ of the homotopy fiber of the assembly map of $K(\Z[G])$.
Following Waldhausen, we call that homotopy fiber the \emph{PL Whitehead spectrum} for the group $G$ with coefficients in $\Z$.
For the applications, it turns out to be even better to perform the same construction with $\Sph$ instead of $\Z$, though it gives the same group at the level of $\pi_0$.

The celebrated \emph{$s$-cobordism theorem} uses this Whitehead group to tell when a given $h$-cobordism from $M$ to $N$ is diffeomorphic to a product $M \times I$.
In fact, when the manifolds have dimension at least 5, the complete obstruction is simply an element of $Wh(\pi_1(M))$.
So when this group vanishes, every $h$-cobordism is trivial, and to construct a diffeomorphism $M \cong N$ it suffices to construct an $h$-cobordism.
This is true in particular when $\pi_1(M)$ is an infinite torsion-free group satisfying the Farrell-Jones conjecture.
On the other hand, when $\pi_1(M)$ is finite, this is often false.
In particular $Wh(\Z/p) \cong \Z^{\frac{p-3}{2}}$ when $p$ is an odd prime (\cite{luckbaum}, Rmk 4).

Theorem \ref{assembly_coassembly_norm} provides a link between these Whitehead groups and Tate cohomology.
It allows us to form the map of cofiber sequences
\begin{equation}\label{map_of_cofib}
\xymatrix{
BG_+ \sma A(*) \ar[r] \ar[d]^-\cong & A(BG) \ar[r] \ar[d]^-{c\alpha \circ \textup{Cartan}} & \Sigma Wh^{PL}(BG) \ar[d] \\
A(*)_{hG} \ar[r]^-{N} & A(*)^{hG} \ar[r] & A(*)^{tG} }
\end{equation}
where $tG$ denotes Tate cohomology with coefficients in a spectrum \cite{greenlees1995generalized}.
This gives the maps of Corollary \ref{cor_second}
\[ \xymatrix{ Wh(G) \ar[r] & \pi_1(A(*)^{tG}) \ar[r] & \pi_1(K(\Z)^{tG}) } \]
and similar maps with coefficients in any ring spectrum $R$.

The behavior of these maps is not yet known, since they appear to be completely new. In the case of $G = \Z/2$, they are reminiscent of the maps from Boardman's proof of Conner and Floyd's five-halves conjecture \cite{boardman1972cobordism}, but more work will be needed to see how closely they are connected.

Here is a second application of Theorem \ref{assembly_coassembly_norm}. Recall that for each prime $p$, there is a sequence of cohomology theories $K(n)$ for $n \geq 0$ that capture the ``pieces'' of the stable homotopy category lying between rational and $p$-local stable homotopy theory. These are the \emph{Morava $K$-theories}; intuitively they separate out pieces of stable homotopy that occur at different frequencies.
One often tries to understand the stable homotopy of a spectrum $X$ by building up knowledge of its localizations $X_{K(n)}$. In $K$-theory the story is even cleaner: if $R$ is the integer ring of a number field, the completion $K(R)^\wedge$ is a connective cover of its $K(1)$-localization $L_{K(1)} K(R)$, so the $K(1)$-localization actually captures all of the important information.

In the $K(n)$-local category, the equivariant norm is always an equivalence.
More precisely, for any spectrum $X$ with a action by a finite group $G$, we get the diagram:
\[ \xymatrix @R=2em{
(X_{hG})_{K(n)} \ar@{-->}[r] \ar[d]^-\sim & (X_{K(n)})^{hG} \ar[d]^-\sim \\
((X_{K(n)})_{hG})_{K(n)} \ar[r]^{N} & ((X_{K(n)})^{hG})_{K(n)} \ar[r] & ((X_{K(n)})^{tG})_{K(n)} \simeq *
} \]
(\cite{hovey1999morava} 8.7)
So the norm becomes an isomorphism in the stable category
\[ \xymatrix{ (X_{hG})_{K(n)} \ar[r]^-\sim & (X_{K(n)})^{hG} } \]
\begin{cor}
Let $G$ be a finite group and $R$ any ring or ring spectrum. Then the assembly map is split injective after $K(n)$-localization on the outside
\[ (BG_+ \sma K(R))_{K(n)} \overset\alpha\ra K(R[G])_{K(n)} \]
and the coassembly map is split surjective after $K(n)$-localization on the inside
\[ G(R[G])_{K(n)} \overset{c\alpha}\ra \Map_*(BG_+,K(R)_{K(n)}) \]
\end{cor}
This includes Corollary \ref{cor_first} from the introduction.

In the case of $R = \Z$, where the $K(1)$-localization is the only useful one, the Cartan map has been well-studied as a means of detecting classes in $K(\Z[G])$.
It has been pointed out to the author that its image in $G(\Z[G])$ is more or less known for $G$ abelian, though it is difficult to find this explicitly worked out.
Our theorem actually suggests that the assembly map cannot see more than $G(\Z[G])$ in this case, at least after $K(1)$-localization.
It underscores the importance of the classical study of $K$-theory through the lens of Swan theory.

In the case of $R = \Sph$, the relevant Swan theory is poorly understood, but the relevant assembly map
\[ BG_+ \sma A(*) \ra A(BG) \]
is quite important for applications to topology.
By the above corollary, this is split injective after $K(n)$-localization, at any prime $p$ and for any $n \geq 0$.
It seems likely that all of these localizations carry nontrivial information.

We would like to be able to conclude something about assembly integrally, or at $p$, but the following issue arises.
Recall that each spectrum $Y$ sits at the top of a tower of $E(n)$-localizations $L_{E(n)} Y$, and the $n$th layer of this tower is detected by the $K(n)$-localization of $Y$.
The spectrum $Y$ has \emph{chromatic convergence} if the limit of this tower is the $p$-localization of $Y$.
Given this language, if we want to deduce something about assembly map for $A$-theory, we should first answer the following question.
\begin{question}
Does $A(X)$ satisfy chromatic convergence? More broadly, is there a clean description of the limit of the chromatic tower under $A(X)$?
\end{question}


\section{Waldhausen categories and $K$-theory.}

In this section we briefly recall from \cite{1126} the definition of algebraic $K$-theory of a Waldhausen category $C$, and compare the standard model of $K(C)$ by the iterated $\mc S_\cdot$ construction to a less standard model by a Segal or May type delooping machine.
These alternative deloopings are essential for our main result.

\subsection{Definitions and the $\mc S_\cdot$ construction.}

\begin{df}
A category $C$ becomes a \emph{Waldhausen category} when it is equipped with two subcategories of cofibrations and weak equivalences, with the following properties.
\begin{itemize}
\item Every isomorphism is both a cofibration and a weak equivalence.
\item There is a chosen zero object $*$ and every object is cofibrant.
\item $C$ has all pushouts along cofibrations. The pushout of a cofibration is a cofibration.
\item (Gluing lemma.) A weak equivalence of homotopy pushout diagrams induces a weak equivalence of pushouts.
\end{itemize}
$C$ has the \emph{saturation axiom} if its weak equivalences satisfy 2 out of 3.
If $C$ has a tensoring over simplicial sets that satisfies a pushout-product axiom, then $C$ has a \emph{cylinder functor} (cf. \cite{1126}).
\end{df}

\begin{df}
Let $C$ be a Waldhausen category.
We define $\mc S_\cdot C$ to be the following simplicial Waldhausen category.
The objects of $\mc S_n C$ are the diagrams in $C$ of the form
\[ \xymatrix @R=1.5em @C=1.5em{
{*} \ar[r] & A_{0,1} \ar[r] \ar[d] & A_{0,2} \ar[r] \ar[d] & \ldots \ar[r] \ar[d] & A_{0,n} \ar[d] \\
& {*} \ar[r] & A_{1,2} \ar[r] \ar[d] & \ldots \ar[r] \ar[d] & A_{1,n} \ar[d] \\
&& {*} \ar[r] & \ldots \ar[r] \ar[d] & \ldots \ar[d] \\
&&& {*} \ar[r] & A_{n-1,n} \ar[d] \\
&&&& {*} } \]
such that the horizontal maps are cofibrations, and the squares are pushouts. This is equivalent to a flag of cofibrations in $C$ of the form
\[ \xymatrix{
{*} \ar[r] & A_{0,1} \ar[r] & A_{0,2} \ar[r] & \ldots \ar[r] & A_{0,n} \\} \]
because each $A_{i,j}$ with $i > 0$ is determined up to canonical isomorphism as the pushout of $*$ and $A_{0,j}$ along $A_{0,i}$.
The face and degeneracy maps are straightforward to define, see \cite[\S 1.3]{1126}.
The morphisms in $\mc S_n C$ are the natural transformations, and the weak equivalences are defined pointwise.
The cofibrations are pointwise cofibrations such that, for each commuting square thus formed, the map from the pushout to the final vertex is also a cofibration (cf. \cite[\S 1.1]{1126}).
\end{df}

\begin{df}
The \emph{algebraic $K$-theory spectrum} $K(C)$ is the symmetric spectrum which at spectrum level $n$ is the realization of the multisimplicial set
\[ |w_\cdot \mc S^{(n)}_\cdot C| \]
Here $w_\cdot$ is shorthand for the nerve on the subcategory of weak equivalences.
One may take a grid of maps in $C$ and add one more dimension, which has only two layers, the first containing only copies of $*$ and the second layer containing the original grid.
This gives the structure maps
\[ |w_\cdot \mc S^{(n)}_\cdot C| \ra \Omega |w_\cdot \mc S^{(n+1)}_\cdot C| \]
It is not hard to check that the space $|w_\cdot \mc S^{(n)}_\cdot C|$ is $(n-1)$-connected.
\end{df}

\begin{thm}[Waldhausen]
These structure maps are weak equivalences when $n \geq 1$. Therefore there is an equivalence of loop spaces
\[ \Omega^\infty K(C) \simeq \Omega |w_\cdot \mc S_\cdot C| \]
\end{thm}

Using (\cite{schwede2008homotopy}, Lem 2.3(ii)) this implies that the homotopy groups of $K(C)$ have bounded filtration, so $K(C)$ is a semistable symmetric spectrum. In other words, the na\" ively defined homotopy groups of $K(C)$ agree with the homotopy groups as measured in the homotopy category of symmetric spectra.

\subsection{Exact functors and the approximation property.}

\begin{df}
A functor $F: C \rightarrow D$ between Waldhausen categories is \emph{exact} if it preserves the zero object, cofibrations, weak equivalences, and pushouts along cofibrations. $F$ has the \emph{approximation property} if it satisfies two conditions:
\begin{itemize}
\item For any $X \overset{f}\rightarrow Y$ in $C$, $f$ is a weak equivalence iff $F(f)$ is a weak equivalence.
\item For any $X \in C$, $Z \in D$, and $F(X) \overset{f}\rightarrow Z$, there is a cofibration $X \rightarrow X'$ in $C$ and weak equivalence $F(X') \overset\sim\rightarrow Z$ in $D$ forming a commuting triangle
\[ \xymatrix{
F(X) \ar[rr] \ar[rd] && Z \\
& F(X') \ar[ru] } \]
\end{itemize}
\end{df}

\begin{thm}[Waldhausen]
If $C$ is a Waldhausen category, then any exact functor $F: C \rightarrow D$ induces a map $K(F): K(C) \rightarrow K(D)$. If $C$ has a cylinder functor and the saturation axiom, and if $F$ satisfies the approximation property, then $K(F)$ is an equivalence of spectra.
\end{thm}

If $M$ is any pointed model category, the subcategory of cofibrant objects $C$ forms a Waldhausen category.
Any left Quillen adjoint $F: M \rightarrow N$ induces an exact functor on the cofibrant objects $C \rightarrow D$, and any left Quillen equivalence has the approximation property.

Unfortunately this always gives $K(C) \simeq *$, because $C$ contains infinite coproducts. 
To fix this we typically restrict to a subcategory $A$ defined by some finiteness condition.
If $A \subset C$ is such a category, closed under weak equivalences and pushouts along cofibrations, and if $B \subset D$ is the subcategory of all objects equivalent to $F(a)$ for some $a \in A$, then $F: A \rightarrow B$ still has the approximation property.
In other words, ``Quillen equivalences give equivalences on $K$-theory,'' but we always have to keep track of the finiteness conditions that define $A$ and $B$.

\subsection{Comparison of Waldhausen, Segal, and May deloopings.}

We conclude this section by discussing deloopings of $K(C) = \Omega|w_\cdot \mc S_\cdot C|$. In essence, we do not need to iterate the $\mc S_\cdot$ construction to get a spectrum; we only need to apply $\mc S_\cdot$ once, and then we may do the rest with a classical delooping machine.

Recall that if $\mc C$ is a symmetric monoidal category, the space $|\mc C| := |N_\cdot \mc C|$ can be delooped infinitely many times.
One approach due to Segal uses $\Gamma$-spaces; another approach due to May uses a monadic bar construction and the Barratt-Eccles operad.
We will use the notation $|\mc C|$, $B^1|\mc C|$, $B^2|\mc C|$, $\ldots$ to refer to the levels of the spectrum produced by either of these two machines.
There is an infinite string of maps
\[ |\mc C| \ra \Omega B^1|\mc C| \qquad B^1|\mc C| \simar \Omega B^2|\mc C| \qquad B^2|\mc C| \simar \Omega B^3|\mc C| \qquad \ldots \]
that are all weak equivalences except the first, and the first map is an equivalence if $|\mc C|$ is grouplike.
Each of these machines may be configured to produce orthogonal spectra \cite{mayequivariant}.

\begin{rmk}
It is common practice to restrict to the permutative subcategory of isomorphisms in $\mc C$ when defining these deloopings.
We have chosen not to state this as part of the machinery, because it is unnecessary.
The above theory requires no assumptions on $|N_\cdot \mc C|$, beyond being the nerve of a permutative category.
\end{rmk}

Returning to the case where $C$ is a Waldhausen category, we observe that $C$ is also a symmetric monoidal category by the coproduct.
We apply the usual rectification trick to make $C$ permutative under this coproduct \cite{maclane1963natural}.
It is easy to see that the subcategory of weak equivalences $wC$ is permutative, and furthermore $\mc S^{(n)}_\cdot C$ and $w\mc S^{(n)}_\cdot C$ are multisimplicial objects in permutative categories.
In other words, the $n$th space of the Waldhausen spectrum $|w_\cdot \mc S^{(n)}_\cdot C|$ is a valid input for the infinite loop space machine $B^{(m)}$.
In fact, when $n \geq 1$ this space is connected, so $B^{(m)}$ gives a true $m$-fold delooping.

\begin{prop}
The spaces $Y_{m,n} = B^{(m)} |w_\cdot \mc S^{(n)}_\cdot C|$ form a bispectrum which is orthogonal in the $m$ direction and symmetric in the $n$ direction.
\end{prop}
\begin{proof}
This is straightforward once we decide how $B$ will act on a loopspace.
If $X$ is an $E_\infty$ space, or a $\Gamma$-space, then $\Omega^k X$ is either an $E_\infty$ space by ``pointwise'' operad action, or a $\Gamma$-space by applying $\Omega^k$ to every level and every map.
In either case, there is a natural interchange
\[ B^{(m)}(\Omega^k X) \ra \Omega^k B^{(m)}(X) \]
which is an equivalence if $X$ is $(k-1)$-connected.
It is now easy to check that $B^{(m)}$ is compatible with the symmetric spectrum structure on the spaces $|w_\cdot \mc S^{(n)}_\cdot C|$.
\end{proof}

As usual, such a bispectrum gives a diagram which commutes up to rearrangements of the loop coordinates:
\[ \xymatrix @R=2em{
|w_\cdot C| \ar[r] \ar[d] & \Omega|w_\cdot S_\cdot C| \ar[r]^-\sim \ar[d]^\sim &
 \Omega^2|w_\cdot S_\cdot^{(2)} C| \ar[r]^-\sim \ar[d]^-\sim & \ldots \\
\Omega B|w_\cdot C| \ar[r] \ar[d]^-\sim & \Omega^2 B|w_\cdot S_\cdot  C| \ar[r]^-\sim \ar[d]^-\sim &
 \Omega^3 B|w_\cdot S_\cdot^{(2)}  C| \ar[r]^-\sim \ar[d]^-\sim & \ldots \\
\Omega^2 B^{(2)}|w_\cdot C| \ar[r] \ar[d]^-\sim & \Omega^3 B^{(2)}|w_\cdot S_\cdot C| \ar[r]^-\sim \ar[d]^-\sim &
 \Omega^4 B^{(2)}|w_\cdot S_\cdot^{(2)} C| \ar[r]^-\sim \ar[d]^-\sim & \ldots \\
\vdots & \vdots & \vdots & 
} \]
The weak equivalences shown follow from easy connectivity arguments.
The left-hand column gives the group completion of $|wC|$ with respect to the sum, and the right-hand region gives the ``group completion'' which splits all cofiber sequences.
In general, these two completions are quite different.

Now, we will make use of the construction in the left-hand column, but it does not currently give the correct stable homotopy type. To remedy this, we apply a shift and a loops in the symmetric-spectrum direction.

\begin{df}
Let $B$ denote the Segal or May machine. To each Waldhausen category $C$ we define the bispectrum $X_{m,n}$ by
\[ X_{m,n} = \Omega B^{(m)} |w_\cdot \mc S^{(n+1)}_\cdot C| \]
This bispectrum is orthogonal in the $m$ slot and symmetric in the $n$ slot.
\end{df}

\begin{rmk}
In an earlier draft we tried simply replacing only the terms in the left-hand column $B^{(m)}|w_\cdot C|$ by $B^{(m)} \Omega |w_\cdot \mc S_\cdot C|$ and leaving the other terms unchanged, but this actually breaks the symmetric spectrum structure. In general if $X_0, X_1, X_2, \ldots$ are the levels of a symmetric spectrum, then $\Omega X_1, X_1, X_2, \ldots$ give a prespectrum in an obvious way, but this prespectrum is not a symmetric spectrum!
\end{rmk}

\begin{thm}\label{alternate_delooping}
The orthogonal spectrum
\[ \{X_{m,0}\}_{m \geq 0} = \{\Omega B^{(m)} |w_\cdot \mc S_\cdot C|\}_{m \geq 0} \]
is naturally equivalent to the prolongation of the symmetric spectrum
\[ \{X_{0,n}\}_{n \geq 0} = \{|w_\cdot \mc S_\cdot^{(n)} C|\}_{n \geq 0} \]
and therefore provides an alternate model of Waldhausen $K$-theory.
\end{thm}

The proof of this theorem reduces to a general statement about bispectra.
Let us recall some notation from \cite{mmss}. Define $\Sigma_S$ as the topological category whose objects are all finite sets (in some universe), and whose morphism spaces are
\[ \Sigma_S(m,n) = (\Sigma_n)_+ \sma_{\Sigma_{n-m}} S^{n-m} \]
Similarly, let $\mathscr J$ be the category whose objects are finite-dimensional inner product spaces in some universe $U \cong \R^\infty$, and whose morphisms are
\[ \mathscr J(V,W) = O(W)_+ \sma_{O(W-V)} S^{W-V} \]
when $V \subset W$ and $W - V$ means orthogonal complement.
Both $\Sigma_S$ and $\mathscr J$ are symmetric monoidal categories, under disjoint union of sets or direct sum of vector spaces, and there is a symmetric monoidal functor $\Sigma_S \rightarrow \mathscr J$ which assigns the set $T$ to the space $\R^T$.
We recall the following facts:

\begin{thm}\cite{mmss}
Diagrams of based spaces over $\Sigma_S$ are equivalent to symmetric spectra. Diagrams of based spaces over $\mathscr J$ are equivalent to orthogonal spectra. The left Quillen equivalence $\mathbb P$ from symmetric to orthogonal spectra is given by left Kan extension along $\Sigma_S \rightarrow \mathscr J$. The smash product of orthogonal spectra $X$ and $Y$ is the left Kan extension of the $\mathscr J \times \mathscr J$-diagram $\sma \circ (X \times Y)$ along the map $\mathscr J \times \mathscr J \rightarrow \mathscr J$ defining the symmetric monoidal structure on $\mathscr J$.
\end{thm}

In this language, Thm \ref{alternate_delooping} reduces to the following proposition.

\begin{prop}
To each diagram $X_{m,n}$ of based spaces over $\mathscr J \times \Sigma_S$ we may functorially assign an orthogonal spectrum $Y$ and a zig-zag of orthogonal spectra
\[ \mathbb{P}\{X_{0,n}\} \ra Y \la \{X_{m,0}\} \]
If $X_{m,n}$ is semistable in the symmetric direction, then this zig-zag is on homotopy groups naturally isomorphic to
\[ \underset{n}\colim \pi_{n+k}(X_{0,n}) \ra \underset{m,n}\colim \pi_{m+n+k}(X_{m,n}) \la \underset{m}\colim \pi_{m+k}(X_{m,0}) \]
So if each row and column of $X_{m,n}$ is an $\Omega$-spectrum, these maps are equivalences.
\end{prop}

\begin{proof}
We make $X_{m,n}$ cofibrant in the projective model structure on $\mathscr J \times \Sigma_S$ diagrams, then define $Y$ to be the left Kan extension of $X_{m,n}$ along the map of categories
\[ \mathscr J \times \Sigma_S \ra \mathscr J \times \mathscr J \ra \mathscr J \]
The above zig-zag is then immediate from the definitions.
It remains to identify the homotopy groups of $Y$.
Prolonging a semistable symmetric spectrum to an orthogonal spectrum does not change level 0, nor does it change the colimit of its homotopy groups, so we may as well assume that $X$ is a bi-orthogonal-spectrum.
Then there is an obvious map from $\underset{m,n}\colim \pi_{m+n+k}(X_{m,n})$ into $\pi_k(Y)$, and our task is to prove it is an isomorphism.
As in \cite{mmss}, 11.9, this reduces to the case where $X$ is a free diagram on a sphere at a single bilevel $(m,n)$.
In this case, the left Kan extension is also isomorphic to a free spectrum at level $m + n$ on a sphere, and so the map on homotopy groups is clearly an isomorphism.
\end{proof}

\section{Waldhausen categories of $G$-spectra and fiberwise spectra.}

In this section we construct seven distinct Waldhausen categories, each with two finiteness conditions, for a total of fourteen examples to which we can apply Waldhausen's $\mc S_\cdot$-construction.
They are all necessary for our proof and applications.
To keep them straight, we introduce some non-standard notation to distinguish them.
In this notation, $B$ is any unbased space, $R$ is any orthogonal ring spectrum, and $G$ is a well-based topological group with the homotopy type of a cell complex. All our spaces are compactly generated (i.e. weak Hausdorff $k$-spaces).
\[ \begin{array}{rl|rl}
\textup{based }G\textup{-sets} & \mc K(G) \\
\textup{based }G\textup{-spaces} & \mc M(G) & R\textup{-modules with }G\textup{ action} & \mc M(G;R) \\
\textup{ex-fibrations over }B & \mc E(B) & R\textup{-module fibrations over }B & \mc E(B;R) \\
\textup{retractive spaces over }B & \mc R(B) & R\textup{-modules over }B & \mc R(B;R) \\
\end{array} \]
The exact functors between these categories point in the following directions:
\[ \xymatrix{
\mc K(G) \ar[d] & \\
\mc M(G) \ar[d]^-E \ar[r]^-{\Sigma^\infty} & \mc M(G;\Sph) \ar[r]^-{R \sma -} \ar[d]^-E & \mc M(G;R) \ar[d]^-E \\
\mc E(BG) \ar[d]^-I \ar[r]^-{\Sigma^\infty} & \mc E(BG;\Sph) \ar[r]^-{R \sma -} \ar[d]^-I & \mc E(BG;R) \ar[d]^-I \\
\mc R(BG) \ar@/^/[u]^-P \ar[r]^-{\Sigma^\infty} & \mc R(BG;\Sph) \ar[r]^-{R \sma -} \ar@/^/[u]^-P & \mc R(BG;R) \ar@/^/[u]^-P \\
} \]
The maps labeled $E$, $I$, and $P$ all preserve $K$-theory.
The categories $\mc R$ and $\mc E$ are used to define assembly and coassembly, respectively. The category $\mc M$ has a good model structure that the others lack. We will lift the assembly and coassembly maps from $\mc R$ and $\mc E$ to $\mc M$, and then to $\mc K$, where they can be analyzed in a more combinatorial way. We will use the categories with $R$ coefficients in the applications.

\subsection{$G$-spaces and fiberwise spaces.}

Throughout this section, $G$ is any well-based topological group with the homotopy type of a cell complex and $B$ is any unbased space.
We will define Waldhausen categories of $G$-spaces and of spaces over $B$, and relate them together when $B = BG$.

Let $\mc T$ be the category of based (compactly generated) topological spaces with the Quillen model structure.
Then the category $G\mc T$ of spaces with a left $G$-action inherits a \emph{projective} model structure, in which the weak equivalences and fibrations are determined by forgetting the $G$-action.
Let $\mc M(G)$ be the subcategory of cofibrant objects; the model structure makes this a Waldhausen category.
We say a $G$-space is \emph{finite} if it is equivalent to one built out of finitely many cells $G \times D^n$, and \emph{finitely dominated} if it is a retract of a finite space in the homotopy category.

\begin{df}
$\mc M_f(G)$ is the Waldhausen category of all finitely dominated cofibrant $G$-spaces.
$\mc M^f(G)$ is the category of all cofibrant $G$-spaces whose underlying nonequivariant space is finitely dominated.
\end{df}

We point out that the $K$-theory of $\mc M_f(G)$ is one model for Waldhausen's $A(BG)$, and $\mc M^f(G)$ for $\uda(BG)$. When $G$ is finitely dominated, the inclusion of categories $\mc M_f(G) \subseteq \mc M^f(G)$ defines the Cartan map $A(BG) \rightarrow \uda(BG)$.

Now fix an unbased space $B$.
Recall that a retractive space over $B$ is a diagram $B \rightarrow X \rightarrow B$ in which the composite is the identity.
We call $X$ the total space.
A map of such retractive spaces $X \rightarrow Y$ is said to be a \emph{weak equivalence} if it is a weak homotopy equivalence on the total spaces.
We will also say that $X$ is \emph{fibrant} if the projection map $X \rightarrow B$ is a Hurewicz fibration.

We recall two notions of cofibration for these retractive spaces.
The map of retractive spaces $A \rightarrow X$ is an $h$-cofibration if the map of total spaces satisfies the homotopy extension property (HEP).
Equivalently, $X \times I$ retracts onto $A \times I \cup X \times 0$, in a non-fiberwise way.
On the other hand, we say $A \rightarrow X$ is an $f$-cofibration if it satisfies the fiberwise homotopy extension property (FHEP).
This means that the retract of $X \times I$ onto $A \times I \cup X \times 0$ may be chosen to respect the map into $B$.
Of course, every $f$-cofibration is an $h$-cofibration.
In this paper we will almost exclusively use $f$-cofibrations, though a result of Kieboom \cite{kieboom1987pullback} appears to imply that we could have gotten away with $h$-cofibrations instead.

\begin{df}
Let $\mc R(B)$ denote the Waldhausen category of $f$-cofibrant retractive spaces over $B$.
Let $\mc E(B)$ the subcategory of such spaces which are fibrant; these are called \emph{ex-fibrations} in the literature.
The subcategories $\mc R_f(B)$ and $\mc E_f(B)$ consist of those $X$ which are a retract up to weak equivalence under $B$ of a finite relative CW complex $B \rightarrow X'$.
The subcategories $\mc R^f(B)$ and $\mc E^f(B)$ consist of those $X$ for which the homotopy fiber of $X \rightarrow B$ is finitely dominated.
\end{df}

It is straightforward to check that these are all Waldhausen categories with cylinder functors.
For $\mc E(B)$ this requires the fact that pushouts of fibrations along $h$-cofibrations are again fibrations \cite{clapp1981duality}.

The reason we have two categories $\mc R(B)$ and $\mc E(B)$ is that they have opposite functoriality.
Let $f: A \rightarrow B$ be any map of base spaces.
Then there is an adjoint pair of functors $f_!$ and $f^*$ defined by the pushout square, resp. pullback square
\[ \xymatrix{ A \ar[r] \ar[d] & B \ar[d] \\ X \ar[r] & f_!X } \qquad \xymatrix{ f^*Y \ar[r] \ar[d] & Y \ar[d] \\ A \ar[r] & B } \]
Since topological spaces are both left and right proper, it follows that $f_!$ is exact on $\mc R(B)$ and $f^*$ is exact on $\mc E(B)$. This allows us to define a covariant and a contravariant functor
\[ \begin{array}{c}
A(B) := K(\mc R_f(B)) \\
\uda(B) := K(\mc E^f(B))
\end{array} \]

\begin{rmk}
Since pushouts and pullbacks are not strictly unique, only unique up to isomorphism, $A(B)$ and $\uda(B)$ are technically not functors. This can be remedied without changing the categories involved up to equivalence; see \cite{raptis2014map}, 3.2.1.
\end{rmk}

\begin{rmk}
The categories $\mc R_f(*)$, $\mc R^f(*)$, $\mc E_f(*)$ and $\mc E^f(*)$ are identical. So we declare that $A(*)$ and $\uda(*)$ are the same spectrum, using any of these models.
\end{rmk}

To define the Cartan map we must relate the categories $\mc E$ and $\mc R$.
The obvious inclusion $I: \mc E(B) \rightarrow \mc R(B)$ is exact, but we need a functor going the other way.
\begin{df}
Define $P: \mc R(B) \rightarrow \mc E(B)$ by taking the $f$-cofibrant retractive space $X$ to the pushout
\[ \xymatrix{
B^I \ar[d]^-\sim_-{\textup{ev}_1} \ar[r] & X \times_B B^I \ar[d]^-\sim \\
B \ar[r] & PX } \]
Here the product $X \times_B B^I$ is taken over the evaluation at 0 map $B^I \rightarrow B$, so that the product is a space over $B$ along the evaluation at 1 map.
\end{df}

Our functor $P$ gives a fibrant replacement for spaces, and later spectra, that is much simpler and more canonical than the ex-fibrant approximation functor found in \cite{ms}, 8.3 and 13.3.

\begin{prop}\label{spaces_homotopy_functor}
$P$ is an exact functor. The functors $I$ and $P$ give inverse equivalences on $K$-theory
\[ K(\mc E_f(B)) \simeq K(\mc R_f(B)) = A(B) \qquad \uda(B) = K(\mc E^f(B)) \simeq K(\mc R^f(B)) \]
Furthermore $A(B)$ and $\uda(B)$ are homotopy functors.
\end{prop}

\begin{proof}
It is elementary to check that $PX$ is an ex-fibration and $P$ preserves weak equivalences.
Moreover, $PX$ preserves cofibrations, because $- \times_B B^I$ preserves cofibrations and there is a pushout square
\[ \xymatrix{
X \times_B B^I \ar[r] \ar[d] & Y \times_B B^I \ar[d] \\
PX \ar[r] & PY } \]
Therefore $P$ is an exact functor.
To check $K(I)$ and $K(P)$ are inverses up to homotopy, it suffices to define natural weak equivalences between $I \circ P$ or $P \circ I$ and the identity functor.
This is done by the natural fiberwise equivalence $X \overset\sim\rightarrow PX$.
To show that $A(B)$ is a homotopy functor, we take a weak equivalence $f: B \rightarrow B'$ and show that the map $K(f_!)$ has left inverse $K(I \circ f^* \circ P)$ and right inverse $K(I \circ f^* \circ P)$ up to homotopy.
To define the homotopies it is enough to define natural fiberwise equivalences
\[ X \overset\sim\ra f^*Pf_!X, \qquad f_!f^*PX \overset\sim\ra PX \overset\sim\la X \]
These come from the equivalence $X \overset\sim\rightarrow PX$ and the adjunction between $f_!$ and $f^*$.
A similar argument works for $\uda(B)$.
\end{proof}

If $\Omega B$ is finitely dominated for every component of $B$, then the object
\[ E(\Omega B) \amalg B \in \mc R_f(B) \]
lies in the category $\mc R^f(B)$.
By induction, all finite complexes $\mc R_f(B)$ are contained in $\mc R^f(B)$.
Therefore the functor $P$ may be viewed as a functor
\[ \xymatrix{ \mc R_f(B) \ar[r]^-P & \mc E^f(B) } \]
This defines the Cartan map $A(B) \rightarrow \uda(B)$.

\begin{rmk}
Our models for $A(B)$ and $\uda(B)$ are equivalent to the more classical models, defined using $h$-cofibrations and maps that are strong homotopy equivalences on the total space.
One may define a zig-zag of exact functors with the approximation property
\[ \mc R(B) \la \mc R''(B) \ra \mc R'(B) \]
in which $\mc R$ is our model, $\mc R'$ is the classical model, and $\mc R''$ is an intermediate model consisting of spaces from $\mc R$ that are strongly homotopy equivalent to a finite complex.
\end{rmk}

Returning to the case where $G$ is a topological group, we let $p: EG \rightarrow BG$ be the principal $G$-bundle with contractible total space.
Then the standard mixing construction and its right adjoint
\[ E(X) = EG \times_G X \qquad F(Y) = \Map_{BG}(EG,Y) \]
define functors $\mc M(G) \rightarrow \mc E(BG)$ and $\mc E(BG) \rightarrow G\mc T$, respectively.
By our assumptions on $G$, the functor $E$ is exact, and the space $E(X)$ is a fiber bundle with fiber $X$.
The functor $F$ is not exact, but if $Y$ is a fibration then $F(Y)$ is weakly equivalent to the fiber of $Y$.
Therefore $X \rightarrow FEX$ is always a weak equivalence, and $EFY \rightarrow Y$ is a weak equivalence if $Y$ is a fibration.
This is enough to show that $E$ has the approximation property, proving

\begin{prop}
The functor $E$ gives equivalences on $K$-theory
\[ K(\mc M_f(G)) \simeq K(\mc E_f(BG)) \qquad K(\mc M^f(G)) \simeq K(\mc E^f(BG)) \]
\end{prop}

\begin{rmk}
Comparisons of this sort are of course not new; see for instance \cite[2.1.5]{1126}. Waldhausen's existing argument constructs an exact functor from spaces over $BG$ to $G$-spaces by tensoring with a principal $G$-bundle, but this does not suit our purposes. We are using the cellular model category of $G$-spaces to allow for easy comparison with model categories of spectra, but we are also avoiding cells in the parametrized setting so that pullbacks can be exact. These choices force us to make the exact functor go from $G$-spaces to retractive spaces, instead of the other way around.
\end{rmk}
Finally, if $G$ is a finite group, we may map the category of finite $G$-sets into $\mc M(G)$.
\begin{df}\label{not_free_df}
Let $\mc K(G)$ be the category of based $G$-sets, with cofibrations the injective maps and weak equivalences the isomorphisms.
Let $\mc N(G)$ be the category of retracts of finite $G$-cell complexes, not necessarily free.
The cofibrations are the retracts of relative $G$-cell complexes and the weak equivalences are the same as $\mc M(G)$.
The subcategories $\mc K^f(G)$ and $\mc N^f(G)$ consist of those spaces which are finitely dominated when the $G$-action is forgotten.
The subcategory $\mc K_f(G)$ consists of the finite free $G$-sets.
\end{df}

We have inclusions $\mc K^f(G) \subseteq \mc N^f(G) \supseteq \mc M^f(G)$, the latter of which gives an equivalence on $K$-theory, and $\mc K_f(G) \subseteq \mc M_f(G)$.
So we may include the $K$-theory of finite sets into the $K$-theory of spaces.
When $G = 1$, $\mc K_f(1) = \mc K^f(1)$ is the category of finite based sets $\mc F$.
Its $K$-theory is the sphere spectrum, by the Barratt-Priddy-Quillen theorem \cite{barratt1972homology}.

\subsection{$R[G]$-module spectra and fiberwise $R$-modules.}

We will construct two models for what could be called the $K$-theory of the group ring $R[G]$, or the $A$-theory of $BG$ with coefficients in $R$.
To define coassembly by a universal property, it is essential to use parametrized spectra, and we show how to set this up while avoiding the model category-theoretic difficulties encountered in \cite{ms}.

Let $R$ be an orthogonal ring spectrum and $G$ a topological group.
We use $R[G]$ as shorthand for the ring spectrum $R \sma G_+$.
Recall that the category of left $R[G]$-module spectra has a cofibrantly generated model structure in which the weak equivalences and fibrations are defined by the forgetful functor to orthogonal spectra \cite{mmss}.
We let $\mc M(G;R)$ be the subcategory of all cofibrant $R[G]$-modules.
Of course, the functor $R \sma -$ is an exact functor from $\mc M(G)$ into $\mc M(G;R)$.
\begin{df}
Let $\mc M_f(G;R)$ be the thick subcategory of modules generated by $R[G]$. Let $\mc M^f(G;R)$ be the subcategory of those modules whose underlying $R$-module is in the thick subcategory of $R$. The $K$-theory and Swan theory of $R[G]$ is simply the Waldhausen $K$-theory of these two categories:
\[ \begin{array}{c}
K(R[G]) := K(\mc M_f(G;R)) \\
G^R(R[G]) := K(\mc M^f(G;R))
\end{array} \]
\end{df}

\begin{rmk}
If $R$ and $G$ are discrete, then $(HR)[G] \simeq H(R[G])$ is a ring spectrum, and the above definition is equivalent to the definitions of $K$-theory or Swan theory using Quillen's $Q$-construction or $BGL_\infty(R)$ and Quillen's plus construction.
\end{rmk}

Now we consider parametrized $R$-module spectra.
Given retractive spaces $X$ and $Y$ over $A$ and $B$, respectively, the \emph{external smash product} $X \barsmash Y$ is the retractive space over $A \times B$ defined by the pushout
\[ \xymatrix{
(X \times B) \cup (A \times Y) \ar[r] \ar[d] & X \times Y \ar[d] \\
A \times B \ar[r] & X \barsmash Y } \]
In particular, if $A = *$ then $X$ is a based space and the external smash product $X \barsmash Y$ is another retractive space over $B$.
This operation tensors and enriches the category of retractive spaces over $B$ over the category of based spaces.
\begin{df}
A \emph{parametrized prespectrum} $X$ over $B$ is a sequence of retractive spaces $X_n$ with structure maps $S^1 \barsmash X_{n-1} \rightarrow X_n$. A \emph{parametrized orthogonal spectrum} is a continuous diagram of retractive spaces indexed by $\mathscr J$.
\end{df}

If $X$ and $Y$ are orthogonal spectra over spaces $A$ and $B$, respectively, the external smash product $X \barsmash Y$ is a parametrized spectrum over $A \times B$ defined by left Kan extension of the $\mathscr J \times \mathscr J$-diagram $\{ X_m \barsmash Y_n \}$ along the direct sum map $\mathscr J \times \mathscr J \rightarrow \mathscr J$.
In other words, it is the Day convolution of $X$ and $Y$ as diagrams over $\mathscr J$.
If $R$ is an orthogonal ring spectrum, a \emph{parametrized $R$-module} is an orthogonal spectrum $X$ equipped with a map $R \barsmash X \rightarrow X$ with the usual associativity properties.

Now we will build a Waldhausen category of parametrized $R$-modules.
It would be convenient to use the existing model structure from \cite{ms}, but pullbacks of the cofibrations are not cofibrations, so $\uda$ would not be a functor.
Instead we use a less cellular notion of cofibration.

\begin{df}
A map of parametrized prespectra $X \rightarrow Y$ is a \emph{(Reedy) cofibration} if in each square
\[ \xymatrix{
S^1 \barsmash X_{n-1} \ar[r] \ar[d] & X_n \ar[d] \\
S^1 \barsmash Y_{n-1} \ar[r] & Y_n } \]
the map from the pushout to $Y_n$ is an $f$-cofibration of retractive spaces over $B$.
When $n = 0$ we interpret $X_{n-1} = B$, so $X_0 \rightarrow Y_0$ must be an $f$-cofibration.
\end{df}

We restrict our attention to parametrized spectra $X$ that are Reedy cofibrant.
This implies that all the structure maps $S^1 \barsmash X_{n-1} \rightarrow X_n$ are $f$-cofibrations, and each level $X_n$ is $f$-cofibrant.
So the fibrant replacement $PX_n$ from the last section is always defined.
It turns to have the following convenient property.
\begin{prop}
For any pair of retractive spaces $X$ and $Y$, there is a homeomorphism $PX \barsmash PY \cong P(X \barsmash Y)$. This is associative and commutative in the usual way, and agrees with the inclusion of $X \barsmash Y$. In other words, both $P$ and the transformation $1 \rightarrow P$ are strong symmetric monoidal.
\end{prop}
\begin{proof}
Suppose $X$ is a space over $B$ and $Y$ is over $B'$. There is an obvious homeomorphism
\[ (X \times_B B^I) \times (Y \times_{B'} (B')^I) \cong (X \times Y) \times_{B \times B'} (B \times B')^I \]
The spaces $PX \barsmash PY$ and $P(X \barsmash Y)$ are both quotients of this space, since a pushout of a quotient map is again a quotient map.
It is easy to check that their equivalence relations coincide.
\end{proof}

It follows formally that if $X$ is a parametrized spectrum, the spaces $PX_n$ assemble into another parametrized spectrum $PX$.
The maps of spaces $X_n \rightarrow PX_n$ give a map of spectra $X \rightarrow PX$ that is a weak equivalence on every spectrum level.
Moreover if $X$ is a parametrized $R$-module spectrum, then so is $PX$, and the map $X \rightarrow PX$ is $R$-linear.

\begin{df}
Let $\mc R(B;R)$ be the category of parametrized $R$-module spectra whose underlying prespectra are Reedy cofibrant.
Let $\mc E(B;R)$ be the further subcategory for which the projections $X_n \rightarrow B$ are also Hurewicz fibrations (and so every level is in $\mc E(B)$).
\end{df}

In both of these categories, we take our cofibrations to be the Reedy cofibrations.
If $X,Y \in \mc E(B;R)$, then a weak equivalence is a map $X \rightarrow Y$ which when restricted to each fiber is a stable equivalence of ordinary spectra.
If instead $X,Y \in \mc R(B;R)$, then a weak equivalence is a map $X \rightarrow Y$ for which $PX \rightarrow PY$ is a stable equivalence on each fiber.
These definitions are consistent because if $X \in \mc E(B;R)$ then the map $X \rightarrow PX$ is an equivalence on every fiber.

These choices make $\mc R(B;R)$ and $\mc E(B;R)$ into Waldhausen categories.
The only nontrivial axiom is the gluing lemma, which is proven using the level equivalences
\[ Y \cup_X Z \ra PY \cup_{PX} PZ \la PY \cup_{PX} PX \barsmash I_+ \cup_{PX} PZ \]
This last construction gives an ex-fibration, and on each fiber it is the ordinary homotopy pushout of spectra, so it preserves all stable equivalences.

The functors $E$ and $F$ from the previous section may be applied to each spectrum level of our $R$-module spectra, giving functors
\[ E: \mc M(G;R) \ra \mc E(BG;R) \qquad F: \mc E(BG;R) \ra R[G]\textup{-Mod} \]
We define the subcategories $\mc R_f(B;R)$ and $\mc E_f(B;R)$ by taking those spectra which are equivalent to $E(X)$ for some $X \in \mc M_f(G;R)$.
Similarly, we let $\mc E^f(B;R)$ be all spectra whose fibers are retracts of the finite $R$-cell modules.
Equivalently, we take those spectra $X$ that are equivalent to some spectrum in the image $E(\mc M^f(G;R))$.

With these notions in place, we have a definition for ``$A$-theory with coefficients in $R$,'' as advertised in the introduction:
\[ \begin{array}{c}
A(B;R) := K(\mc R_f(B;R)) \\
\uda(B;R) := K(\mc E^f(B;R))
\end{array} \]

It is straightforward to verify that $I_+ \barsmash X$ gives a cylinder functor.
The functors $I$, $P$, and $E$ then give equivalences on $K$-theory by the same arguments as before, so we get equivalences
\[ \begin{array}{c}
K(R[G]) = K(\mc M_f(G;R)) \simeq K(\mc R_f(BG;R)) = A(BG;R) \\
G^R(R[G]) = K(\mc M^f(G;R)) \simeq K(\mc E^f(BG;R)) = \uda(BG;R)
\end{array} \]

If $f: A \rightarrow B$ is any map of base spaces, then $f_!$ and $f^*$ make $A(B;R)$ and $\uda(B;R)$ into homotopy functors. We list a few of the details of this argument. The functors $f_!$ and $f^*$ are defined on spectra by taking the pushout or pullback along $f$ on each spectrum level. They give $R$-module spectra in a canonical way because $f_!$ commutes with fiberwise suspension and $f^*$ commutes with fiberwise loops. It is easy to check that both $f_!$ and $f^*$ preserve cofibrations, and that $f^*$ preserves weak equivalences. It is also true that $f_!$ preserves weak equivalences, but this relies on the following result of Clapp.
Recall that a parametrized spectrum $Z$ is an \emph{$\Omega$-spectrum} if every level is a fibration and the adjoint structure maps $Z_{n-1} \rightarrow \Omega_B Z_n$ are all homeomorphisms.

\begin{prop}[\cite{clapp1984homotopy}]
The inclusion of $\Omega$-spectra into all spectra has a left adjoint $L$, called spectrification, which on the Reedy cofibrant spectra is given by the usual construction $(LX)_n = \colim_k \Omega_B^k X_{n+k}$.
Therefore $[X,Y] \cong [LX,Y]$ if $Y$ is an $\Omega$-spectrum.
If $X$ and $Y$ are Reedy cofibrant $\Omega$-spectra with levels homotopy equivalent to relative CW complexes, then $X \rightarrow Y$ is a weak equivalence iff it is a homotopy equivalence of parametrized spectra.
\end{prop}

\begin{cor}
Suppose $X$ and $Y$ are Reedy cofibrant and each level is homotopy equivalent to a relative CW complex. Then $X \rightarrow Y$ is an equivalence of parametrized prespectra iff it induces a bijection $[Y,Z] \rightarrow [X,Z]$ for all $\Omega$-spectra $Z$.
\end{cor}

Exactness of $f_!$ follows immediately, because of the natural isomorphism $[f_!X,Z] \cong [X,f^*Z]$, and the observation that $f^*Z$ is always an $\Omega$-spectrum. Once we know that $f_!$ and $f^*$ are exact, the proof that they are homotopy functors is the same as the proof in Proposition \ref{spaces_homotopy_functor}.

\subsection{Pairings.}\label{pairing_section}

Next we will construct the external pairings on $A(B;R)$ and $\uda(B;R)$, similar to those found in \cite{weiss2000products,williams2000bivariant}. This requires us to use a more complicated and restrictive definition of ``cofibration.'' Our cofibrations must include the images of $R[G]$-module cofibrations under \mbox{$EG \times_G -$}, be preserved by pushouts and pullbacks, and interact well with smash products. To our knowledge, there is no notion of ``cofibration'' in the literature that fits the bill, and we address this by establishing one such notion for parametrized orthogonal spectra. The definition we give is in fact very close to the notion of a ``flat cofibration'' or ``$S$-cofibration'' of symmetric spectra \cite[5.3.6]{hovey2000symmetric}, \cite[III, \S 2]{schwede2012symmetric}.

\begin{df}
If $X$ is a parametrized orthogonal spectrum, then regard $X$ as a diagram $\mathscr J \rightarrow Top_B$. Define the \emph{$n$-skeleton} $\sk^n X$ by restricting to the objects of $\mathscr J$ which have dimension at most $n$, and then taking an enriched left Kan extension back to all of $\mathscr J$. Define the \emph{$n$th latching object} $L_n X$ to be $(\sk^{n-1} X)_n$.
\end{df}

Concretely, the $m$th level of $\sk^n X$ is given by the coequalizer
\[ \bigvee_{i \leq j \leq n} X_i \barsmash \mathscr J(i,j) \sma \mathscr J(j,m) \rightrightarrows \bigvee_{i \leq n} X_i \barsmash \mathscr J(i,m) \ra (\sk^n X)_m \]
where $\bigvee$ refers to coproduct of retractive spaces, or union along $B$.
Of course, if $m \leq n$ then $(\sk^n X)_m \cong X_m$. If $m \geq n$, then $(\sk^n X)_m$ is a quotient of $X_n \barsmash_{O(n)} \mathscr J(n,m)$.

\begin{df}
A map $X \rightarrow Y$ of parametrized orthogonal spectra is an \emph{orthogonal (Reedy) cofibration} if in each square
\[ \xymatrix{
L_n X \ar[r] \ar[d] & X_n \ar[d] \\
L_n Y \ar[r] & Y_n } \]
the map from the pushout to $Y_n$ is an $O(n)$-equivariant $f$-cofibration of retractive spaces over $B$.
In other words, there is a fiberwise retract of $Y_n$ onto an appropriate subspace, and this map \emph{also} respects the $O(n)$-action.
\end{df}

Clearly the orthogonal Reedy cofibrations are closed under pushouts, transfinite compositions, and retracts. We will show that they are generated by semi-free spectra on maps of spaces $K \rightarrow L$ having the $O(n)$-equivariant fiberwise homotopy extension property.

\begin{df}
A \emph{semi-free} orthogonal spectrum $F_n^\triangleleft K$ on a based $O(n)$-space $K$ is a spectrum which at level $m$ is $K \sma_{O(n)} \mathscr J(n,m)$. The functor $F_n^\triangleleft(-)$ is the left adjoint of the forgetful functor that sends an orthogonal spectrum $X$ to its $n$th space $X_n$ with the $O(n)$-action. These definitions still make sense if $K$ is a retractive space with a fiberwise $O(n)$-action.
\end{df}

\begin{lem}
If $K \rightarrow L$ is an $O(n)$-equivariant $f$-cofibration then $F_n^\triangleleft K \rightarrow F_n^\triangleleft L$ is an orthogonal Reedy cofibration, and on level $m$ it is an $O(m)$-equivariant $f$-cofibration.
\end{lem}

\begin{proof}
The relevant map is only nontrivial at spectrum level $n$, where it is the $O(n)$-cofibration $K \rightarrow L$. Therefore $F_n^\triangleleft K \rightarrow F_n^\triangleleft L$ is an orthogonal Reedy cofibration. It is easy to check that the map
\[ K \barsmash_{O(n)} A \ra L \barsmash_{O(n)} A \]
has the $O(m)$-equivariant homotopy extension property for any $O(n) \times O(m)$-space $A$. In particular, this applies to the map $F_n^\triangleleft K \rightarrow F_n^\triangleleft L$ at level $m$, which is
\[ K \barsmash_{O(n)} \mathscr J(n,m) \ra L \barsmash_{O(n)} \mathscr J(n,m) \]
\end{proof}

\begin{df}
If $X \rightarrow Y$ is a map of parametrized orthogonal spectra, define its \emph{$n$-skeleton} as the pushout
\[ \xymatrix{
\sk^n X \ar[r] \ar[d] & X \ar[d] \\
\sk^n Y \ar[r] & \sk^n(X \ra Y) } \]
Clearly the map $X \rightarrow Y$ is filtered by these skeleta, $Y$ is the colimit of the skeleta.
\end{df}

\begin{lem}\label{easy_pushout}
For each map of parametrized orthogonal spectra $X \rightarrow Y$ there is a natural pushout square
\[ \xymatrix{
F_n^\triangleleft (X_n \cup_{L_n X} L_n Y) \ar[r] \ar[d] & F_n^\triangleleft Y_n \ar[d] \\
\sk^{n-1}(X \ra Y) \ar[r] & \sk^n(X \ra Y) } \]
\end{lem}

\begin{proof}
Just compare the universal properties.
\end{proof}

\begin{cor}
The class of orthogonal Reedy cofibrations is the smallest class of maps that is closed under retracts, pushouts, and transfinite compositions, and containing $F_n^\triangleleft K \rightarrow F_n^\triangleleft L$ for any $O(n)$-equivariant $f$-cofibration $K \rightarrow L$.
\end{cor}

\begin{cor}
If $X \rightarrow Y$ is an orthogonal Reedy cofibration, then 
each map $X_n \rightarrow Y_n$ is an $O(n)$-equivariant $f$-cofibration.
\end{cor}

This guarantees that we can take the strict cofiber of $X \rightarrow Y$ and it will have the correct stable homotopy type. It also ensures that pushouts along orthogonal cofibrations will behave the way we expect, allowing for $A(B)$ to be a functor.

Our main technical result for these cofibrations is that they satisfy a pushout-product axiom. This appears to be a new result even in the non-fiberwise case of $B = *$. It suggests that it is possible to further loosen the notion of a ``flat spectrum'' as in \cite{schwede2012symmetric} while preserving the fact that the smash product of two flat spectra preserves weak equivalences.

\begin{prop}\label{pushout_product}
The pushout-product of two orthogonal cofibrations, $K \rightarrow X$ over $A$ and $L \rightarrow Y$ over $B$, is an orthogonal cofibration over $A \times B$. If $X$ is cofibrant and $Y \rightarrow Y'$ is a weak equivalence of cofibrant spectra then $X \barsmash Y \rightarrow X \barsmash Y'$ is a weak equivalence.
\end{prop}

\begin{proof}
First we show that a pushout-product of cofibrations is a cofibration.
Since our cofibrations are generated by maps of the form $F_n^\triangleleft K \rightarrow F_n^\triangleleft L$, it suffices to take a pushout-product of two such maps.
We observe that the product of two semi-free diagrams over $\mathscr J$ gives a semi-free diagram over $\mathscr J \times \mathscr J$, and therefore
\[ F_m^\triangleleft X \barsmash F_n^\triangleleft Y \cong F_{m+n}^\triangleleft (O(m+n)_+ \barsmash_{O(m) \times O(n)} X \barsmash Y) \]
This construction sends pushout-products of spaces to pushout-products of spectra, so it suffices to show that if $K \rightarrow X$ is an $O(m)$-equivariant $f$-cofibration over $A$ and $L \rightarrow Y$ is an $O(n)$-equivariant $f$-cofibration over $B$, the pushout-product is an $O(m) \times O(n)$-cofibration over $A \times B$.
The retraction we want is the one given by the usual formula for showing that a pushout-product of NDR-pairs is an NDR-pair, see\cite{steenrod1967convenient}, Thm 6.3.
It is easy to check that this formula preserves fiberwise and equivariant maps, so we are done.

Next we check that smashing with a cofibrant spectrum $X$ preserves weak equivalences between cofibrant spectra $Y \rightarrow Y'$.
It suffices to prove this inductively for $\sk^n X$. If $Y$ is cofibrant, the map $(\sk^{n-1} X) \barsmash Y \rightarrow (\sk^{n} X) \barsmash Y$ is a pushout-product of two cofibrations, so it is a cofibration and its strict cofiber is a homotopy cofiber. Therefore we only have to examine effect of smashing with the cofiber of $\sk^{n-1} X \rightarrow \sk^n X$.

So without loss of generality, $X$ is a semi-free spectrum $F_n^\triangleleft K$, and we need to show that
\[ (F_n^\triangleleft K \sma Y)_m = (K \barsmash Y_{m-n}) \barsmash_{O(n) \times O(m-n)} O(m)_+ \]
preserves stable equivalences in the $Y$ variable when $Y$ is cofibrant.
Each level of this construction is a smash product over $O(n) \times O(m-n)$ with the free $O(n) \times O(m-n)$-CW complex $O(m)$, so it preserves levelwise weak equivalences.
So, without loss of generality we can assume that $K$ is a free $O(n)$-CW complex, the levels $Y_{m-n}$ are $O(m-n)$-CW complexes, and all quotients by a group action are homotopy quotients.

The above spectrum may be rewritten as the smash product of a parametrized prespectrum $\textup{sh}_{-n} Y$ and the retractive space $K$ over $O(n)$:
\[ \textup{sh}_{-n} Y \barsmash_{O(n)} K \]
\[ (\textup{sh}_{-n} Y)_m = Y_{m-n} \underset{O(m-n)}\barsmash O(m)_+ \]
We choose to think of this as a prespectrum with a free $O(n)$-action.
If $\textup{sh}_{-n}$ preserves stable equivalences, then it is easy to check that tensoring over $O(n)$ with a free $O(n)$-CW complex $K$ preserves stable equivalences, so our construction will preserve all weak equivalences in the $Y$ variable.

We are reduced to showing that $\textup{sh}_{-n}$ sends weak equivalences of cofibrant orthogonal spectra to weak equivalences of prespectra.
We will prove that the map of parametrized prespectra given at level $m$ by
\[ Y_{m-n} \ra Y_{m-n} \underset{O(m-n)}\barsmash O(m)_+ \]
is a stable equivalence.
We filter $Y$ by its skeleta $\sk^k Y$ and observe that, as before, we only need to prove this is true on the cofiber of $\sk^{k-1} Y \rightarrow \sk^k Y$, which is a semi-free spectrum $F_k^\triangleleft L$ on a retractive space $L$ with an $O(k)$-action.
The space
\[ (F_k^\triangleleft L)_{m-n} = (S^{m-n-k} \barsmash L) \barsmash_{O(m-n-k) \times O(k)} O(m-n)_+ \]
has connectivity $m-n-k-1$ and so the map
\[ (F_k^\triangleleft L)_{m-n} \ra (F_k^\triangleleft L)_{m-n} \underset{O(m-n)}\barsmash O(m)_+ \]
has connectivity $2(m-n)-k-1$. This connectivity increases faster than $m$, so this map is an equivalence of prespectra. This finishes the induction up the skeleta of $Y$, so $\textup{sh}_{-n}$ does indeed preserve stable equivalences.
\end{proof}

\begin{cor}
If $R$ is orthogonal cofibrant then $E(X) = EG \times_G X$ sends every cofibration of $R[G]$-module spectra to an orthogonal cofibration of parametrized $R$-module spectra.
\end{cor}

\begin{proof}
A free cell of orthogonal $R[G]$-module spectra
\[ R[G] \sma F_k S^{n-1}_+ \ra R[G] \sma F_k D^n_+ \]
is sent under $E$ to
\[ (EG \amalg BG) \barsmash R \sma F_k S^{n-1}_+ \ra (EG \amalg BG) \barsmash R \sma F_k D^n_+ \]
and this is the external smash product of a free cell of orthogonal spectra with a cofibrant parametrized orthogonal spectrum.
\end{proof}

With these lemmas, it is straightforward to verify that the categories $\mc R(B;R)$ and $\mc E(B;R)$ from the previous section have all of the same properties if we take instead the orthogonal cofibrant objects and the orthogonal cofibrations between them.
The least obvious property is the fact that $f_!$ preserves weak equivalences, but $f_!$ already preserves level equivalences between spectra whose levels are $f$-cofibrant, so we may replace our orthogonal cofibrant spectra with cofibrant prespectra and then use the same proof.

Since we are changing the definitions, we will also take this opportunity to make a standard technical modification so that the pairings below are strictly associative.
We allow the Waldhausen categories $\mc R(B;R)$ and $\mc E(B;R)$ to also include objects consisting of
\begin{itemize}
\item a $k$-tuple of ring spectra $R_1, \ldots, R_k$,
\item an isomorphism of rings $R_1 \sma \ldots \sma R_k \cong R$,
\item a $k$-tuple of parametrized modules $M_1, \ldots, M_k$ over the spaces $B_1, \ldots, B_k$,
\item and a homeomorphism $B_1 \times \ldots \times B_k \cong B$.
\end{itemize}
In practice, all but one of these rings will be the sphere spectrum, and all but one of these spaces will be $*$.
For the purpose of defining the morphisms, we treat the above object as if it were the external smash product $M_1 \barsmash \ldots \barsmash M_k$, which is an $R$-module over $B$.
Clearly these new categories are equivalent to the old ones and so they give homotopy equivalent $K$-theory.

Now we are ready to define our pairings.

\begin{df}
A pairing of Waldhausen categories, or biexact functor, is a bifunctor $F: C \times D \rightarrow E$ that is exact in each variable separately, such that for every choice of cofibration $a \rightarrow b$ in $C$ and cofibration $x \rightarrow y$ in $D$, in the square of cofibrations
\[ \xymatrix{
F(a,x) \ar[r] \ar[d] & F(b,x) \ar[d] \\
F(a,y) \ar[r] & F(b,y) } \]
the map $F(a,y) \cup_{F(a,x)} F(b,x) \rightarrow F(b,y)$ is also a cofibration.
\end{df}

The following is a consequence of \cite{blumberg2011derived}, Thm 2.6:
\begin{prop}
A pairing of Waldhausen categories $C_1 \times C_2 \rightarrow D$ induces a map of spectra $K(C_1) \sma K(C_2) \rightarrow K(D)$ in a natural way. Given four pairings making this diagram of functors commute
\[ \xymatrix{
C_1 \times C_2 \times C_3 \ar[r] \ar[d] & D_1 \times C_3 \ar[d] \\
C_1 \times D_2 \ar[r] & E } \]
the two induced maps $K(C_1) \sma K(C_2) \sma K(C_3) \rightarrow K(E)$ are identical.
\end{prop}

\begin{prop}\label{pairings}
If $B$ and $B'$ are unbased spaces, $R$ and $S$ ring spectra, then there are pairings of symmetric spectra
\[ \xymatrix @R=0.4em { A(B;R) \sma A(B';S) \ar[r] & A(B \times B'; R \sma S) \\
 \uda(B;R) \sma \uda(B';S) \ar[r] & \uda(B \times B'; R \sma S) } \]
which are natural with respect to all pairs of maps of unbased spaces and pairs of maps of rings. They have the obvious associativity relation in the case of three base spaces and three rings. They commute up to homotopy with the Cartan map when it is defined.
\end{prop}

\begin{proof}
We send a parametrized $R$-module $M$ over $B$ and an $S$-module $N$ over $B'$ to their external smash product $M \barsmash N$ as defined above. By Prop \ref{pushout_product} this defines biexact functors
\[ \xymatrix @R=0.3em { \mc R(B;R) \times \mc R(B';S) \ar[r] & \mc R(B \times B'; R \sma S) \\
\mc E(B;R) \times \mc E(B';S) \ar[r] & \mc R(B \times B'; R \sma S) } \]
We have to check that the second pairing actually lands in $\mc E$; in other words, a smash product of level-fibrant spectra is level-fibrant.
Recall that ex-fibrations are preserved under pushouts along $f$-cofibrations and colimits along $f$-cofibrations (8.2.1 in \cite{ms}).
Therefore on the space level, an external smash product of ex-fibrations is an ex-fibration.
If $X$ and $Y$ are level-fibrant spectra over $A$ and $B$, respectively, we check that $X \barsmash Y$ is level-fibrant by showing inductively that $(\sk^n X) \barsmash Y$ is level-fibrant.
By the pushout square of Lemma \ref{easy_pushout}, we may assume that $X$ is semi-free on an ex-fibration $K$ over $A$ with a fiberwise $O(n)$-action.
We are reduced to showing that the space
\[ (F_n^\triangleleft K \sma Y)_m = (K \barsmash Y_{m-n}) \barsmash_{O(n) \times O(m-n)} O(m)_+ \]
is an ex-fibration.
Since $K \barsmash Y_{m-n}$ is already an ex-fibration, this is an easy induction up the free $O(n) \times O(m-n)$-cells of $O(m)$.

To deal with associativity, we modify the above functor up to natural isomorphism (which does not change biexactness).
We instead take a $k$-tuple of modules and an $l$-tuple of modules to the obvious $(k+l)$-tuple, and do not actually smash any of them together.
Now our rule is strictly associative, not just associative up to natural isomorphism.

Finally we check the finiteness conditions. Biexactness means that cofiber sequences in each variable are sent to cofiber sequences, so if we wish to check that a pair of thick subcategories is sent to a given thick subcategory, we only have to check the generators. It's easy to see that the external smash product of spaces sends $B$ with a single cell attached and $B'$ with a single cell attached to $B \times B'$ with a single cell attached, so we get
\[ \xymatrix @R=0.4em { \mc R_f(B;R) \times \mc R_f(B';S) \ar[r] & \mc R_f(B \times B'; R \sma S) } \]
giving the pairing on $A$-theory.
Even easier, since the external smash product is on each fiber just the smash product of the fibers, we get
\[ \xymatrix @R=0.4em { \mc E^f(B;R) \times \mc E^f(B';S) \ar[r] & \mc E^f(B \times B'; R \sma S) } \]
giving the pairing on $\uda$-theory.
Note that these pairings do not commute with the functor $P$ defined in the last section on the nose, only up to equivalence.
\end{proof}

We remark in passing that these pairings make topological Swan theory into a ring, and $K$-theory into a module over that ring, just as in the classical case.
\begin{cor}
For any space $B$, $\uda(B)$ is a ring spectrum, and $A(B)$ and $A(B;R)$ are $\uda(B)$-modules.
\end{cor}

We are interested in the following special cases.
Taking one of our spaces to be the one-point space $*$, we get natural pairings
\[ \xymatrix @R=0.3em { K(R) \sma A(B) \ar[r] & K(R) \sma A(B;\Sph) \ar[r] & A(B;R) \\
 K(R) \sma \uda(B) \ar[r] & K(R) \sma \uda(B;\Sph) \ar[r] & \uda(B;R) } \]
This will allow us to reduce the proof of Theorem \ref{assembly_coassembly_norm} to the case $R = \Sph$.

\begin{rmk}
Our construction of $A(B;R)$ and $\uda(B;R)$ is not natural with respect to all maps of ring spectra, only those maps $R \rightarrow S$ for which $S$ is cofibrant as an $R$-module. One might be able to modify our definition of ``cofibration'' even further to get naturality for all maps of rings, without breaking the biexactness we proved above.
\end{rmk}

\section{Assembly and coassembly of $R$-modules.}\label{assembly}

In this section we briefly recall both the precise definitions and universal properties of assembly and coassembly maps, cf. \cite{weiss1993assembly,williams2000bivariant}.

\begin{df}
If $X_\cdot$ is a simplicial set, the \emph{category of simplices} $\Delta_{X_\cdot}$ is a small category whose objects are $\coprod_{p \geq 0} X_p$. The morphisms from $x \in X_p$ to $x' \in X_q$ consists of all injective maps of simplicial sets $\Delta[p]_\cdot \rightarrow \Delta[q]_\cdot$ making the following square commute.
\[ \xymatrix{
\Delta[p]_\cdot \ar[d] \ar[r]^-{x} & X_\cdot \ar@{=}[d] \\
\Delta[q]_\cdot \ar[r]^-{x'} & X_\cdot } \]
\end{df}

If $X$ is an unbased space, its singular simplices $S_\cdot X$ form a simplicial set.
Each object in $\Delta_{S_\cdot X}$ gives a continuous map $\Delta^p \rightarrow X$, and the morphisms between such objects are the factorizations $\Delta^p \rightarrow \Delta^q \rightarrow X$ through compositions of face maps.

We will make use of the \emph{last vertex map} $|N_\cdot \Delta_{X_\cdot}| \overset\sim\rightarrow |X_\cdot|$.
It is most easily defined in the special case where $X_\cdot$ is the nerve of a category $C$.
In that case, we define the last vertex map on the flag of face maps
\[ \Delta[p_0] \ra \Delta[p_1] \ra \ldots \ra \Delta[p_k] \ra N_\cdot C \]
by regarding them as functors
\[ [p_0] \ra [p_1] \ra \ldots \ra [p_k] \ra C \]
where $[p]$ is the poset of integers from $0$ to $p$.
We let $f_i$ denote the composite functor $[p_i] \rightarrow C$, and we send this flag to the $k$-simplex in $N_k C$
\[ \xymatrix @C=0.9in { f_0(p_0) \ar[r]^-{f_1(f_0(p_0) \rightarrow p_1)} & f_1(p_1) \ar[r]^-{f_2(f_1(p_1) \rightarrow p_2)} & \ldots \ar[r]^-{f_k(f_{k-1}(p_{k-1}) \rightarrow p_k)} & f_k(p_k) } \]
It is straightforward to check that this agrees with faces and degeneracies and so gives a well-defined map of simplicial sets.
\[ N_\cdot \Delta_{N_\cdot C} \ra N_\cdot C \]
Furthermore it is natural with respect to functors in $C$, and so taking $C = [n]$ we get a collection of maps of simplicial sets
\[ N_\cdot \Delta_{\Delta[n]} \ra \Delta[n] \]
which are natural with respect to maps $[n] \rightarrow [m]$ in the simplicial category $\mathbf{\Delta}$.
Since any simplicial set $X_\cdot$ is a colimit of these standard simplices, the definition of the last-vertex map extends to all $X_\cdot$, and it is always a weak equivalence, using Lemma A.5 from \cite{segal1974categories}.

Now we may define the assembly and coassembly maps.
Recall that a functor $F$ is a \emph{homotopy functor} if it sends weak equivalences to weak equivalences.
Our main examples of interest are the covariant functors $A(B)$ and $A(B;R)$ and the contravariant functors $\uda(B)$, and $\uda(B;R)$.

\begin{df}
If $F$ is any covariant homotopy functor from unbased spaces to spectra, the \emph{assembly map} of $F(X)$ is the zig-zag
\[ X_+ \sma F(*) \overset\sim\la |N_\cdot \Delta_{S_\cdot X}|_+ \sma F(*) \cong \underset{(\Delta^p \rightarrow X) \in \Delta_{S_\cdot X}}\hocolim F(*) \overset\sim\la \underset{\Delta^p \rightarrow X}\hocolim F(\Delta^p) \ra F(X) \]
\end{df}

\begin{df}
If $F$ is any contravariant homotopy functor from unbased spaces to spectra, the \emph{coassembly map} of $F(X)$ is the zig-zag
\[ F(X) \ra \underset{(\Delta^p \rightarrow X) \in \Delta_{S_\cdot X}^\op}\holim F(\Delta^p) \overset\sim\la \underset{\Delta^p \rightarrow X}\holim F(*) \cong \Map_*(|N_\cdot \Delta_{S_\cdot X}|_+, F(*)) \overset\sim\la \Map_*(X_+,F(*)) \]
\end{df}

The assembly and coassembly maps are characterized by a universal property. Recall that a homotopy functor $F$ is \emph{linear} if $F(\emptyset) \simeq *$ and $F$ takes homotopy pushout squares to homotopy pushout/pullback squares of spectra. The \emph{homotopy category} of functors is obtained by inverting the natural transformations of functors that induce a weak equivalence $F(X) \overset\sim\rightarrow G(X)$ for every space $X$. These definitions are unchanged if $F$ is a contravariant functor.

\begin{prop}
Assume that $F(\emptyset) \simeq *$. If $F$ is covariant, then the assembly map is the universal approximation of $F$ on the left by a linear functor in the homotopy category of functors. If $F$ is contravariant, the coassembly map is the universal approximation of $F$ on the right by a linear functor.
\end{prop}

In fact, the proof of this is quite formal and follows the method of (\cite{calc3}, 1.8). It is also possible to modify the definitions of assembly and coassembly so as to drop the requirement that $F(\emptyset) \simeq *$, or to make higher-order polynomial approximations to $F$; see \cite{malkiewich2015tower}.

It turns out that assembly and coassembly play well with multiplicative structure.
If $F$ is a functor into $R$-modules, then the assembly and coassembly maps are $R$-module maps.
If $F$ lands in ring spectra then the coassembly map is a map of rings, and if $F$ lands in coalgebra spectra then the assembly map is a map of coalgebras.
These facts motivated our proof of Theorem \ref{assembly_coassembly_norm}, but the final form of the proof only requires this simple observation:

\begin{prop}\label{obvious}
If $M$ is a spectrum and $F$ a covariant functor, the assembly map for the functor $M \sma F$ is the smash product of the identity of $M$ and the assembly map for $F$. If $F$ is contravariant then the adjoint of the coassembly map
\[ X_+ \sma (M \sma F(X_+)) \ra M \sma F(*) \]
is the smash of the identity of $M$ and the coassembly map of $F$.
\end{prop}

\section{A combinatorial lift of assembly and coassembly.}\label{combinatorial}

We have defined the assembly map for $A$-theory as a morphism $[\alpha]$ in the stable homotopy category.
In this section we lift $[\alpha]$ to an explicit map of spectra
\[ BG_+ \sma A(*) \overset\alpha\ra A(BG) \]
Our description will agree with the more classical notion of ``assembly'' by the units of a ring.
We then proceed to do the same thing for coassembly, resulting in an apparently new definition of coassembly that does not appeal to a universal property.

\begin{df}
Let $G$ be a finite group. By abuse of notation, let $G$ refer also to the category with one object, whose set of morphisms (with composition from right to left) is the group $G$. Let $\ti G$ refer to the category whose objects are the elements of the group $G$, and each ordered pair of objects has a unique isomorphism connecting them. We will draw the arrows of these categories from right to left. When taking the nerve, we apply the usual conventions for the face and degeneracy maps, so for example $d_0$ deletes the object on the left:
\[ d_0( \bullet \overset{g_1}\la \bullet \overset{g_2}\la \ldots \overset{g_k}\la \bullet ) = \bullet \overset{g_2}\la \ldots \overset{g_k}\la \bullet \]
Then we define $BG = |N_\cdot G|$ and $EG = |N_\cdot \ti G|$.
\end{df}

Of course, these are isomorphic to the usual definitions of $BG$ and $EG$.
We think of the arrow $g \leftarrow h$ in $\ti G$ as multiplication on the left by $gh^{-1}$, since this description is invariant under the obvious right $G$-action on the category $\ti G$. This labeling of the arrows of $\ti G$ defines a functor $\ti G \rightarrow G$, which on realizations is the familiar map $EG \rightarrow EG/G \cong BG$.

It will be necessary to make some precise statements about monodromy, so let us fix our conventions here.
When $E \ra BG$ is a covering space, and $F$ is the fiber over the 0-simplex of $BG$, we define the monodromy left action of $G$ on $F$ as follows.
Take the 1-simplex $\Delta^1 \in BG$ given by
\[ \bullet \overset{g}\la \bullet \]
and take the map $F \times \Delta^1 \rightarrow E$ which on the right-hand end of $\Delta^1$ is the inclusion of $F$ into $E$.
The left-hand end maps to $F$ back into $F$, but in a nontrivial way, and we declare this to be the action of $g$ on $F$.
The reader can check that this defines a left action, and that canonically $E \cong EG \times_G F$ and $F \cong \Map_{BG}(EG,E)$.

Now we may return to our formula for the assembly map.
\begin{df}
Define the map of spectra
\[ BG_+ \sma K(\mc M_f(*)) \overset\alpha\ra K(\mc M_f(G)) \]
as the map of bisimplicial spaces
\[ (N_p G)_+ \sma w_p \mc S_q \mc M_f(*) \ra w_p \mc S_q \mc M_f(G) \]
\[ g_1,\ldots,g_n; \xymatrix{ X_{0,1} \ar[r] \ar[d]_-\sim^-{w_{1,1}} & X_{0,2} \ar[r] \ar[d]_-\sim^-{w_{1,2}} & \ldots \\
X_{1,1} \ar[r] \ar[d]_-\sim^-{w_{2,1}} & X_{1,2} \ar[r] \ar[d]_-\sim^-{w_{2,2}} & \ldots \\
\vdots & \vdots & }
\mapsto \xymatrix{ X_{0,1} \sma G_+ \ar[r] \ar[d]_-\sim^-{w_{1,1} \sma (- \cdot g_1)} & X_{0,2} \sma G_+ \ar[r] \ar[d]_-\sim^-{w_{1,2} \sma (- \cdot g_1)} & \ldots \\
X_{1,1} \sma G_+ \ar[r] \ar[d]_-\sim^-{w_{2,1} \sma (- \cdot g_2)} & X_{1,2} \sma G_+ \ar[r] \ar[d]_-\sim^-{w_{2,2} \sma (- \cdot g_2)} & \ldots \\
\vdots & \vdots & } \]
and similarly for iterates of the $\mc S_\cdot$-construction. Here the horizontal direction is the $\mc S_\cdot$ direction, viewed as a flag of cofibrations. The quotients are suppressed from the notation, but the map is defined on them by the same rule.
\end{df}

\begin{prop}\label{combinatorial_assembly}
The following diagram commutes up to homotopy, and therefore $\alpha$ defines the assembly map for $A(BG)$:
\[ \xymatrix{
BG_+ \sma A(*) \ar[r]^-\alpha & K(\mc M_f(G)) \ar[r]_-\sim^-{K(I \circ E)} & K(\mc R_f(BG)) \\
|N_\cdot \Delta_{N_\cdot G}|_+ \sma A(*) \ar[u]^-\sim_-{\textup{last vertex}} && \underset{\Delta^p \rightarrow X}\hocolim A(\Delta^p) \ar[ll]_-\sim \ar[u] } \]
\end{prop}

\begin{proof}
We define an explicit simplicial homotopy between the two legs of the diagram
\begin{equation}\label{assembly_proof}
\xymatrix{
(N_k G)_+ \sma w_k \mc R_f(*) \ar[r]^-{K(I \circ E) \circ \alpha} & w_k \mc R_f(BG) \\
(N_k \Delta_{N_\cdot G})_+ \sma w_k \mc R_f(*) \ar[u]_-{\textup{last vertex}} & B_k(w_k \mc R_f(\Delta^{-}),\Delta_{N_\cdot G},*) \ar[l] \ar[u] }
\end{equation}
where $B_\bullet$ refers to the categorical bar construction.
The proof is then finished by applying $\mc S_\cdot$ as many times as necessary to define the homotopy on level $n$ of the $K$-theory spectrum.

A $k$-simplex in the lower-right corner of the diagram (\ref{assembly_proof}) is given by a flag of categories
\[ [p_0] \ra [p_1] \ra \ldots \ra [p_k] \ra G \]
and a flag of weak equivalences of retractive spaces over $\Delta^{p_0}$
\[ X_0 \overset{w_1}\ra X_1 \overset{w_2}\ra \ldots \overset{w_k}\ra X_k \]
The long route of the diagram pushes these spaces forward along $r: \Delta^{p_0} \ra *$ to get a flag of based spaces
\[ r_!X_0 \overset{r_!w_1}\ra r_!X_1 \overset{r_!w_1}\ra \ldots \overset{r_!w_1}\ra r_!X_k \]
and selects the $k$-simplex in the nerve of $G$
\[ \bullet \overset{g_1}\la \bullet \overset{g_2}\la \ldots \overset{g_k}\la \bullet \]
where $g_i$ is the image in the category $G$ of the unique arrow in the poset $[p_i]$ connecting the image of the last vertex of $[p_{i-1}]$ to the last vertex of $[p_i]$.
This brings us to the top-left corner of the diagram; the final step transforms this into the single flag of retractive spaces over $BG$
\[ (r_!X_0) \barsmash EG^+ \overset{w_1 \times \cdot g_1}\ra (r_!X_1) \barsmash EG^+ \overset{w_2 \times \cdot g_2}\ra \ldots \overset{w_k \times \cdot g_k}\ra (r_!X_k) \barsmash EG^+ \]
where $EG^+ = EG \amalg BG$ is the bundle $EG \rightarrow BG$ with an extra basepoint section.
On the other hand, the short route of (\ref{assembly_proof}) takes our original data to the flag
\[ i_!X_0 \overset{i_!w_1}\ra i_!X_1 \overset{i_!w_2}\ra \ldots \overset{i_!w_k}\ra i_!X_k \]
where $i: \Delta^{p_0} \rightarrow BG$ is the inclusion induced by the functor $[p_0] \rightarrow G$.
It is enough to define a commuting diagram of weak equivalences of retractive spaces over $BG$
\[ \xymatrix @C=0.5in{
i_!X_0 \ar[r]^-{i_!w_1} \ar[d]^-{f_0} & i_!X_1 \ar[r]^-{i_!w_2} \ar[d]^-{f_1} & \ldots \ar[r]^-{i_!w_k} & i_!X_k \ar[d]^-{f_k} \\
(r_!X_0) \barsmash EG^+ \ar[r]^-{w_1 \times \cdot g_1} & (r_!X_1) \barsmash EG^+ \ar[r]^-{w_2 \times \cdot g_2} & \ldots \ar[r]^-{w_k \times \cdot g_k} & (r_!X_k) \barsmash EG^+
} \]
that agrees with deletion or duplication of elements in each flag, and that is natural with respect to the $X_j$ and $w_j$ (to get compatibility with the $\mc S_\cdot$ construction).
We define $f_j$ as the product of the identity map of $X_j$ and the composite
\[ \xymatrix{ X_j \ar[r] & \Delta^{p_0} \ar[r] & \Delta^{p_j} \ar[r] & EG } \]
Here the first map is the projection map of the retractive space $X_j$, the second map comes from our flag of categories, and the final map is the unique lift of $\Delta^{p_j} \ra BG$ to $EG$ by sending the final vertex of $[p_j]$ to the object of $\ti G$ labeled by $1 \in G$.
This sends the basepoint section of $X_j$ to the basepoint section of $(r_!X_j) \barsmash EG^+$, so $f_j$ is well-defined on the quotient $i_!X_j$.
With this definition, the commutativity of the $j$th square above boils down to the commutativity of the square of categories
\[ \xymatrix{
[p_{j-1}] \ar[r] \ar[d] & [p_j] \ar[d] \\
\ti G \ar[r]^-{\cdot g_j} & \ti G } \]
where the vertical maps are our lifts sending the final vertex to 1.
This square commutes by our definition of $g_j$.
\end{proof}

Now that we have established a combinatorial model of the assembly map, we turn our attention to the coassembly map.
Recall the category $\mc N^f(G)$ from Definition \ref{not_free_df}.
\begin{df}
Define the map of spectra
\[ K(\mc N^f(G)) \overset{c\alpha}\ra \Map_*(BG_+,K(\mc M^f(*))) \]
whose adjoint is given by the map of bisimplicial spaces
\[ (N_p G)_+ \sma w_p \mc S_q \mc N^f(G) \ra w_p \mc S_q \mc M^f(*) \]
\[ g_1,\ldots,g_n; \xymatrix{ X_{0,1} \ar[r] \ar[d]_-\sim^-{w_{1,1}} & X_{0,2} \ar[r] \ar[d]_-\sim^-{w_{1,2}} & \ldots \\
X_{1,1} \ar[r] \ar[d]_-\sim^-{w_{2,1}} & X_{1,2} \ar[r] \ar[d]_-\sim^-{w_{2,2}} & \ldots \\
\vdots & \vdots & }
\mapsto \xymatrix{ X_{0,1} \ar[r] \ar[d]_-\sim^-{g_1^{-1} \cdot w_{1,1}} & X_{0,2} \ar[r] \ar[d]_-\sim^-{g_1^{-1} \cdot w_{1,2}} & \ldots \\
X_{1,1} \ar[r] \ar[d]_-\sim^-{g_2^{-1} \cdot w_{2,1}} & X_{1,2} \ar[r] \ar[d]_-\sim^-{g_2^{-1} \cdot w_{2,2}} & \ldots \\
\vdots & \vdots & } \]
and similarly for iterates of the $\mc S_\cdot$-construction.
\end{df}

\begin{prop}\label{combinatorial_coassembly}
The following diagram commutes up to homotopy, and therefore $c\alpha$ defines the coassembly map for $\uda(BG)$:
\[ \xymatrix{
K(\mc M^f(G)) \ar[r]^-\sim \ar[d]_-\sim^-{K(E)} & K(\mc N^f(G)) \ar[r]^-{c\alpha} & F(BG_+,A(*)) \ar[d]_-\sim^-{\textup{last vertex}} \\
K(\mc E^f(BG)) \ar[r] & \underset{\Delta^p \rightarrow X}\holim \uda(\Delta^p) \ar@{<-}[r]^-\sim & F(|N_\cdot \Delta_{N_\cdot G}|_+,A(*)) } \]
\end{prop}

\begin{proof}
This proof is in many ways dual to the previous one.
It is enough to define an explicit simplicial homotopy between the two legs of the diagram
\begin{equation}\label{coassembly_proof}
\xymatrix{
w_k \mc M^f(G) \ar[r]^-{c\alpha} \ar[d] & F((N_k G)_+, w_k \mc R^f(*)) \ar[d]^-{\textup{last vertex}} \\
C_k(*,\Delta_{N_\cdot G},w_k \mc E^f(\Delta^{-})) & F((N_k \Delta_{N_\cdot G})_+, w_k \mc R^f(*)) \ar[l] }
\end{equation}
where $C_\bullet$ refers to the categorical cobar construction.
Once this is accomplished, the proof is finished by applying $\mc S_\cdot$ as many times as necessary to define the homotopy on level $n$ of the $K$-theory spectrum.
Though the target spectrum in our holim system is not fibrant, both of our maps factor through this one, so after composing with fibrant replacement they are still homotopic.

A $k$-simplex in the upper-left corner of (\ref{coassembly_proof}) is given by a flag of coarse weak equivalences of spaces with a left $G$-action
\[ Y_0 \overset{w_1}\ra Y_1 \overset{w_2}\ra \ldots \overset{w_k}\ra Y_k \]
Given this and a flag of simplices
\[ [p_0] \ra [p_1] \ra \ldots \ra [p_k] \ra G \]
the long route of (\ref{coassembly_proof}) gives the flag of retractive spaces over $\Delta^{p_k}$
\[ \Delta^{p_k} \times Y_0 \overset{\id \times (g_1^{-1} \cdot w_1)}\ra \Delta^{p_k} \times Y_1 \overset{\id \times (g_2^{-1} \cdot w_2)}\ra \ldots \overset{\id \times (g_k^{-1} \cdot w_k)}\ra \Delta^{p_k} \times Y_k \]
Here $g_i$ has the same definition as before.
Note that the maps $g_i^{-1} \cdot w_i(-)$ and $w_i(g_i^{-1} \cdot -)$ are identical because $w_i$ is equivariant.
The short route of our diagram ends with the flag of retractive spaces
\[ (i^*EG) \times_G Y_0 \overset{\id \times w_1}\ra (i^*EG) \times_G Y_1 \overset{\id \times w_2}\ra \ldots \overset{\id \times w_k}\ra (i^*EG) \times_G Y_k \]
where $i$ is the inclusion $\Delta^{p_k} \rightarrow BG$ coming from our functor $[p_k] \ra G$.
Next we define homeomorphisms $f_j$ of retractive spaces over $\Delta^{p_k}$
\[ \xymatrix @C=0.5in{
\Delta^{p_k} \times Y_0 \ar[r]^-{g_1^{-1} \cdot w_1} \ar[d]^-{f_0} & \Delta^{p_k} \times Y_1 \ar[r]^-{g_2^{-1} \cdot w_2} \ar[d]^-{f_1} & \ldots \ar[r]^-{g_k^{-1} \cdot w_k} & \Delta^{p_k} \times Y_k \ar[d]^-{f_k} \\
(i^*EG) \times_G Y_0 \ar[r]^-{\id \times w_1} & (i^*EG) \times_G Y_1 \ar[r]^-{\id \times w_2} & \ldots \ar[r]^-{\id \times w_k} & (i^*EG) \times_G Y_k
} \]
which agree with deletion or duplication of the spaces $Y_j$.
To define $f_j$, we take the unique lift of $\Delta^{p_j} \rightarrow BG$ to $EG$ which takes the final vertex to 1.
This extends uniquely to a lift of $\Delta^{p_k}$ which takes the final vertex to $g_k^{-1}g_{k-1}^{-1}\ldots g_{j+1}^{-1}$.
It follows that this square commutes:
\[ \xymatrix @C=6em{ 
\Delta^{p_k}_{j-1} \times Y_{j-1} \ar[r]^-{\id \times g_j^{-1}\cdot w_j} \ar[d] & \Delta^{p_k}_j \times Y_j \ar[d] \\
EG \times_G Y_{j-1} \ar[r]^-{\cdot g_j \times g_j^{-1}\cdot w_j} & EG \times_G Y_j } \]
Furthermore, the $g_j$ and $g_j^{-1}$ along on the bottom row cancel out, leaving us with $\id \times w_j$.
So our definition of $f_j$ has the correct properties, and we are done.
\end{proof}

Our combinatorial model for the assembly and coassembly maps fit into a strictly commuting diagram
\[ \xymatrix{
BG_+ \sma K(\mc F) \ar[d] \ar[r]^-\alpha & K(\mc K_f(G)) \ar[d] \ar[r] & K(\mc K^f(G)) \ar[d] \ar[r]^-{c\alpha} & F(BG_+,K(\mc F)) \ar[d] \\
BG_+ \sma K(\mc M_f(*)) \ar[r]^-\alpha & K(\mc M_f(G)) \ar[d]^-{K(E)}_-\sim \ar[rd] \ar[r] & K(\mc N^f(G)) \ar[r]^-{c\alpha} & F(BG_+,K(\mc M^f(*))) \\
& K(\mc E_f(BG)) \ar[rd] & K(\mc M^f(G)) \ar[d]^-{K(E)}_-\sim \ar[u]^-\sim & \\
& K(\mc R_f(BG)) \ar[r]_-{\textup{Cartan}} \ar[u]^-\sim_-{K(P)} & K(\mc E^f(BG)) & } \]
The maps along the top row are the obvious restrictions of $\alpha$ and $c\alpha$ from all finite based spaces to finite based sets.
The remaining unlabeled maps are all induced by inclusions of Waldhausen categories.
Therefore we get a commuting diagram in the homotopy category
\[ \xymatrix{
BG_+ \sma K(\mc F) \ar[d] \ar[r]^-\alpha & K(\mc K_f(G)) \ar[d] \ar[r] & K(\mc K^f(G)) \ar[d] \ar[r]^-{c\alpha} & F(BG_+,K(\mc F)) \ar[d] \\
BG_+ \sma A(*) \ar[r]^-{\textup{assembly}} & A(BG) \ar[r]^-{\textup{Cartan}} & \uda(BG) \ar[r]^-{\textup{coassembly}} & F(BG_+,A(*)) } \]
This gives us a strategy for proving Theorem \ref{assembly_coassembly_norm}.
It is enough to prove that our combinatorial model for assembly gives the norm on $K(\mc F)$, because then we can deduce the general case using the pairings of Proposition \ref{pairings}.

\section{Proof that the lift is the norm.}

In this section we recall and provide some results on equivariant transfer and norm maps.
Then we prove that the composite of the assembly and coassembly maps, on the $K$-theory of finite sets, is the norm.
Finally, we use this to finish the proof of the main theorem.

In this section the term ``$G$-spectrum'' refers to an orthogonal spectrum with a $G$-action, or an $\Sph[G]$-module.
We let $f$ denote any fibrant replacement functor in this model category.
Our constructions could easily be interpreted as taking place in the model category of genuine $G$-spectra from \cite{mandell2002equivariant}, but we do not need that interpretation here.

If $G$ is a finite group, we think of it as a left $(G \times G)$-set with action $(g,h)k = hkg^{-1}$.
We pick any equivariant embedding of $G$ into a $(G \times G)$-representation $V$, and then Pontryagin-Thom collapse to produce a map
\[ S^0 \ra \Omega^V(S^V \sma G_+) \]
Here $S^V$ is the one-point compactification of $V$, and $\Omega^V$ denotes the space of maps out of $S^V$.
We add trivial representations to $V$ to make this a map of spectra:
\[ \Sph \ra \Omega^V \Sigma^V \Sigma^\infty_+ G \overset\sim\la \Sigma^\infty_+ G \]
This zig-zag is the \emph{pretransfer} $[\tau]$. It is well-defined as a map in the homotopy category of $(G \times G)$-spectra.

If $X$ is a cofibrant $\Sph[G]$-module, we smash the pretransfer over $G$ with the identity of $X$.
This gives the \emph{equivariant transfer map} $[\tau(X)]$:
\[ X_{hG} \simeq \Sph \sma_G X \ra (\Omega^V \Sigma^V \Sigma^\infty_+ G) \sma_G X \overset\sim\la \Sigma^\infty_+ G  \sma_G X \cong X \]
This is still a map of $G$-spectra, when we give $X_{hG}$ the trivial $G$-action.
Therefore $[\tau(X)]$ has an equivariant lift into the fibrant replacement $fX$, and this lift must factor through the homotopy fixed points $(fX)^{hG}$.
That factorization is the \emph{equivariant norm map} of $X$:
\[ N(X): X_{hG} \ra (fX)^{hG} \]
We will occasionally drop the fibrant replacement $f$ from the notation, when the meaning is clear.

\begin{rmk}
This definition is translated from \cite{lewis1986equivariant}, III.7.
It only gives a natural transformation between $-_{hG}$ and $-^{hG}$ in the homotopy category, though it could be modified to be natural with respect to any small collection of maps.
In \cite{reich2014adams} this definition is modified to give a natural transformation on the entire category of spectra.
\end{rmk}

We first prove that the norm can simplify if $X$ is a suspension spectrum $\Sigma^\infty_+ E$, and $p: E \rightarrow B$ is a principal $G$-bundle but with $G$ acting on the left.
Recall that $E$ may be embedded equivariantly into $U \times B$ over $B$, when $U$ is a countable-dimensional real inner product space $U$ with an orthogonal $G$-action.
Then the Pontryagin-Thom collapse gives a map in the equivariant homotopy category
\[ \Sigma^\infty B_+ \overset{\theta(p)}\ra f\Sigma^\infty E_+ \]
It is well-known that this agrees with the equivariant transfer $[\tau(\Sigma^\infty_+ E)]$, so we omit the proof.
Now if $h: E \rightarrow K$ is an equivariant map to a space $K$ with a trivial $G$-action, we consider the composite
\[ \Sigma^\infty_+ B \overset{N}\ra (f\Sigma^\infty_+ E)^{hG} \ra (f\Sigma^\infty_+ K)^{hG} = \Map(BG,f\Sigma^\infty_+ K) \]

\begin{prop}\label{nonequivariant_norm}
This is adjoint to a Pontryagin-Thom collapse along the bundle
\[ EG \times_G E \ra BG \times B \]
followed by the map $EG \times_G E \rightarrow K$ coming from $h$.
\end{prop}

To avoid confusion with left and right actions, we always assume $EG = B(*,G,G)$ has a right $G$-action.
We let $EG^\ell$ denote the same space, but with a left $G$-action given by composing the existing right $G$-action with the inverse map $g \mapsto g^{-1}$. It is then easy to identify the quotient $(EG^\ell \times E)/G$ with the balanced product $EG \times_G E$ as spaces over $BG$.

\begin{proof}
We use an explicit model for the Pontryagin-Thom collapse following \cite{cohen2004multiplicative}.
Let $S^V_\epsilon = V / (V - B_\epsilon(0))$, for any $G$-representation $V$ and $\epsilon > 0$.
The identity of $V$ induces an equivariant homotopy equivalence $S^V \ra S^V_\epsilon$.
We use $\Sigma^V_\epsilon X$ as shorthand for $S^V_\epsilon \sma X$, and $\Omega^U \Sigma^U_\epsilon$ as shorthand for the colimit of $\Omega^V \Sigma^V_\epsilon$ over all inclusions of representations.
Then we define the Pontryagin-Thom collapse
\[ B_+ \overset{\theta(p)}\ra \Omega^U \Sigma^U_\epsilon E_+ \]
by the formula $(b,u) \mapsto (u - e,e)$, with $e$ the preimage of $b$ closest to $u$.
To give a simple formula for the norm, we fill the dotted line in the square
\[ \xymatrix{
B \ar[r]^-{\theta(p)} \ar@{-->}[d] & \Map^G(EG^\ell,\Omega^U \Sigma^U_\epsilon E_+) \ar[d]^-\sim \\
\Map^G(EG^\ell,\Omega^{\R^\infty} \Sigma^{\R^\infty}_\epsilon E_+) \ar[r]^-\sim & \Map^G(EG^\ell,\Omega^{\R^\infty \oplus U} \Sigma^{\R^\infty \oplus U}_\epsilon E_+) } \]
by choosing an equivariant map $\phi$ of $EG^\ell$ into the space of fiberwise embeddings of $E$ into $\R^\infty \times B$ whose fibers are always at least $\epsilon$ apart.
This latter space has the usual conjugation left $G$-action, and it is weakly contractible, so $\phi$ exists and is unique up to equivariant homotopy.
We rewrite $\phi$ as a map $E \times EG^\ell \ra \R^\infty$ and notice that $\phi(x,y) = \phi(gx,gy)$.
Then we take the dotted map to be the adjoint of the Pontryagin-Thom collapse
\[ B \times EG^\ell \ra \Omega^{\R^\infty} \Sigma^{\R^\infty}_\epsilon E_+ \]
which for each point $y \in EG^\ell$ collapses onto the image of $\phi(-,y)$.
So these two branches give
\[ (s,u,b,y) \in \R^\infty \times U \times B \times EG^\ell \mapsto (s, u - \ti e, \ti e) \in \Omega^{\R^\infty \oplus U} \Sigma^{\R^\infty \oplus U}_\epsilon E_+ \]
\[ (s,u,b,y) \in \R^\infty \times U \times B \times EG^\ell \mapsto (s - \phi(\ti e,y), u, \ti e) \in \Omega^{\R^\infty \oplus U} \Sigma^{\R^\infty \oplus U}_\epsilon E_+ \]
where $\ti e$ is in the first case the point in $E \subset U$ closest to $U$, and in the second case the point of $\phi(-,y)$ closest to $s$.

We can define an explicit, $G$-equivariant homotopy between these two maps by
\[ (s - (\cos t)\phi(\ti e,y), u - (\sin t)\ti e, \ti e) \in \Omega^{\R^\infty \oplus U} \Sigma^{\R^\infty \oplus U}_\epsilon E_+ \]
Here $\ti e$ is always chosen so that the first two quantities are smaller than $\epsilon$, and if this is not possible, we go to the basepoint instead.
This provides a homotopy between the norm and a much simpler map, which when composed with $h$ gives
\[ (s,b,y) \in \R^\infty \times B \times EG^\ell \mapsto (s - \phi(\ti e,y), h(\ti e)) \in \Omega^{\infty} \Sigma^{\infty}_\epsilon K_+ \]
This is invariant under the left $G$-action on $EG^\ell$ and so descends to a map on $BG$.
We reinterpret $\phi$ as a fiberwise embedding of the bundle $EG \times_G E \rightarrow BG \times B$ into $\R^\infty \times BG \times B$, and then the composite is clearly a Pontryagin-Thom collapse followed by $h$.
\end{proof}

Taking the special case $E = EG^\ell$ and $K = *$, the norm for the sphere spectrum
\[ \Sigma^\infty_+ BG \ra \Map_*(BG_+,\Omega^\infty S^\infty_\epsilon) \]
is adjoint to a Pontryagin-Thom collapse along the bundle
\[ EG \times_G EG^\ell \cong (EG \times EG)/G \rightarrow BG \times BG \]
followed by collapsing the total space of that bundle to a point, as in Thm 8 of \cite{lewis1982classifying}.
It is easy to verify that this bundle is equivalent to the diagonal map $BG \rightarrow BG \times BG$.
In a somewhat non-symmetric way, we identify its fiber with $G$ along the isomorphism $G \rightarrow (G \times G)/G$ sending $\gamma$ to $(\gamma,1)$.
Then the left $(G \times G)$-action on the fiber is given by left and right multiplication, $(g,h)\gamma = g\gamma h^{-1}$.

Next we describe how the norm simplifies when $X$ has a trivial $G$-action.
The proof follows the previous proposition, as the $X$ term remains inert at every step.
For simplicity, we suppress the fibrant replacements in the final answer.

\begin{prop}\label{norm_of_trivial_action}
If $X$ has trivial $G$-action, the transfer for $EG^\ell_+ \sma X$ is the smash of the transfer for $\Sigma^\infty_+ EG^\ell$ and the identity of $X$. The norm is adjoint to the smash of the identity of $X$ and the Pontryagin-Thom collapse along $(EG \times EG)/G$, followed by the collapse of $(EG \times EG)/G$ to a point.
\end{prop}

The reader is invited to compare this norm map to the composite of assembly and coassembly on the $K$-theory of finite sets:
\[ (BG \times BG)_+ \sma K(\mc F) \ra K(\mc F) \]
Informally, this map uses the infinite loop space structure on $K(\mc F)$ to add a given point to itself $G$ times.
This sum has a monodromy as we rove around $BG \times BG$ that agrees with the bundle described in Proposition \ref{norm_of_trivial_action}, so we expect this composite map to be a Pontryagin-Thom collapse as well.
We will now make this idea precise.

Let $\mc C$ be a $\Sigma$-free $E_\infty$ operad of unbased spaces, with May's convention of $\mc C(0) = *$.
Let $X$ be a $\mc C$-algebra, $Y$ an ordinary based space, and $f: Y \rightarrow X$ a map of based spaces.
Let $B^\infty X$ be the (essentially unique) spectrum whose zeroth space is $X$, so that $f$ must come from some map in the stable homotopy category $\overline{f}: \Sigma^\infty Y \rightarrow B^\infty X$.
Let $B$ be an unbased space with fundamental group $\Gamma$, and $\Gamma \overset\phi\rightarrow \Sigma_j$ a homomorphism.
Let $E \rightarrow B$ be the $j$-sheeted covering space whose monodromy is given by $\phi$.
(Since all of our monodromy actions are left actions, we assume that the composition of permutations in $\Sigma_n$ is being written from right to left.)
Finally, let $E\Sigma_j \overset{i}\rightarrow \mc C(j)$ be any $\Sigma_j$-equivariant map, necessarily an equivalence.
Consider the composite
\[ \xymatrix{ B \times Y \ar[r] & B\Gamma \times X \cong E\Gamma \times_\Gamma X \ar[r]^-{E\phi \times \Delta} & E\Sigma_j \times_{\Sigma_j} X^j \ar[r]^-{i \times \id} & \mc C(j) \times_{\Sigma_j} X^j \ar[r] & X } \]
If the point in $B \times Y$ is of the form $(b,*)$ then its image in $X$ is the basepoint, so this can be interpreted as a map out of $B_+ \sma Y$.
Intuitively, this composite takes each point $y \in Y$ to a sum of $j$ copies of $f(y) \in X$, but as we rove around $B$, the ordering in this sum is permuted according to the rule given by $\Gamma \rightarrow \Sigma_j$.

\begin{prop}\label{transfer_via_operad}
In the homotopy category, this composite is adjoint to the map of spectra
\[ \xymatrix{ \Sigma^\infty B_+ \sma Y \ar[r]^-{\theta(p) \sma \id} & \Sigma^\infty E_+ \sma Y \ar[r]^-{r \sma \overline{f}} & B^\infty X } \]
where $\theta(p)$ is Pontryagin-Thom collapse, and $r: E \rightarrow *$ is the collapse of $E$ onto a point.
\end{prop}

This theorem is apparently quite classical (cf. \cite{kahn1972applications} and \cite{adams1978infinite}) so we will give only a brief sketch of the proof.

\begin{proof}
It suffices to prove this for one fixed $E_\infty$ operad, since any two are related by a zig-zag of weak equivalences of operads $\mc C \rightarrow \mc C'$.
We take $\mc C$ to be the little $\infty$-cubes operad.
Our composite map from $B \times Y$ into $X$ factors through $\mc C(j) \times_{\Sigma_j} X^j$.
This gives for each point of $B \times Y$ a 1-simplex in the space $\Omega^\infty B(\Sigma^\infty,C,X)$, arising from some finite level $\Omega^n B(\Sigma^n,C_n,X)$.
That family of 1-simplices defines a homotopy of maps $B \times Y \rightarrow \Omega^\infty B^\infty X$, which at one end is our original composite included along $X \overset\sim\rightarrow \Omega^\infty B^\infty X$, and which at the other end instead maps
\[ \mc C(j) \times_{\Sigma_j} X^j \ra \Omega^\infty \Sigma^\infty X \]
using the map of monads $C \rightarrow \Omega^\infty \Sigma^\infty$.
One may write an explicit homotopy between this latter map and a Pontryagin-Thom collapse along the covering space $E \rightarrow B$ with the monodromy we described.
Therefore the above composite is homotopic to the composite
\[ \xymatrix{
B \times Y \ar[r]^-{1 \times f} & B \times X \ar[r]^-{\theta_p} & \Omega^\infty \Sigma^\infty_\epsilon (E \times X)_+ \ar[r] & \Omega^\infty \Sigma^\infty_\epsilon X \ar@{<-}[r]^\sim & \Omega^\infty \Sigma^\infty X \ar[r] & \Omega^\infty B^\infty X \\ } \]
Checking basepoints, we rearrange this into the composite map stated in the proposition.
\end{proof}

We can now prove Thm \ref{assembly_coassembly_norm} from the introduction.
\begin{thm}
If $G$ is a finite group and $R$ is a ring spectrum then the composite of assembly, Cartan, and coassembly
\[ \xymatrix @C=0.6in{ BG_+ \sma K(R) \ar[r]^-{\textup{assembly}} & A(BG;R) \ar[r]^-{\textup{Cartan}} & \uda(BG;R) \ar[r]^-{\textup{coassembly}} & F(BG_+,K(R)) } \]
is the equivariant norm map
\[ K(R)_{hG} \ra K(R)^{hG} \]
on $K(R)$ with the trivial $G$-action.
\end{thm}

\begin{proof}
Take $n = |G|$ and fix a bijection between $G$ and the standard set of $n$ elements.
Define the homomorphism $\phi: \Gamma = G \times G \rightarrow \Sigma_n$ by sending the pair $(g,h)$ to the permutation $x \mapsto gxh^{-1}$.

Let $\mc C$ be the Barratt-Eccles operad, so that $\mc C(n) = E\Sigma_n = |N_\cdot \ti \Sigma_n|$.
If $C$ is any permutative category, the action map
\[ |N_\cdot \ti \Sigma_n| \times |N_\cdot C|^n \ra |N_\cdot C| \]
can be defined the 0-skeleta by
\[ \sigma, c_1, \ldots, c_n \mapsto c_{\sigma^{-1}(1)} \vee c_{\sigma^{-1}(2)} \vee \ldots \vee c_{\sigma^{-1}(n)} =: S_\sigma \]
and on $|N_\cdot \ti \Sigma_n|$ times the $0$-skeleton of $|N_\cdot C|^n$ by sending $(\sigma_1,\ldots,\sigma_k;\sigma)$ to the flag
\[ S_{\sigma_1\sigma_2\ldots\sigma_k\sigma} \overset{\sigma_1^{-1}}\ra S_{\sigma_2\ldots\sigma_k\sigma} \overset{\sigma_2^{-1}}\ra
\ldots \overset{\sigma_{k-1}^{-1}}\ra S_{\sigma_k\sigma} \overset{\sigma_k^{-1}}\ra S_{\sigma} \]
where each map shuffles the summands according to the permutation listed above the arrow.
This extends to the higher skeleta of $|N_\cdot C|^n$, by allowing ourselves to put flags of $k$ composable nontrivial maps into each of the $c_i$ slots.

By Propositions \ref{combinatorial_assembly} and \ref{combinatorial_coassembly}, the composite of assembly and coassembly on the $K$-theory of finite sets is, when restricted to $Y = S^0$, the map of spaces
\[ (BG \times BG)_+ \ra K(\mc F) \simeq \Omega^\infty S^\infty \]
defined simplicially by
\[ \begin{array}{c}
N_k G \times N_k G \ra w_k \mc F \\
(g_1,\ldots,g_k;h_1,\ldots,h_k) \mapsto \xymatrix @C=0.5in{ G_+ \ar[r]^-{g_1^{-1} \cdot - \cdot h_1} & G_+ \ar[r]^-{g_2^{-1} \cdot - \cdot h_2} & \ldots \ar[r]^-{g_k^{-1} \cdot - \cdot h_k} & G_+ }
\end{array} \]
and then we apply the usual inclusion $|w_\cdot \mc F| \rightarrow \Omega|w_\cdot \mc S_\cdot \mc F|$.

Under our choice of homomorphism $\Gamma = G \times G \ra \Sigma_n$, this agrees with the operad action of $E\Sigma_n$ on $|w_\cdot \mc F|$ that we described above.
(Technically, the two agree up to a natural isomorphism in $w_\cdot \mc F$, so the two maps are homotopic, not identical.)
This allows us to rewrite our assembly and coassembly composite as
\[ \xymatrix @C=5em{ B\Gamma \cong E\Gamma/\Gamma \ar[r]^-{E\phi \times (\Delta \circ f)} & E\Sigma_n \times_{\Sigma_n} \Omega |w_\cdot \mc S_\cdot \mc F|^n \ar[r] & \Omega |w_\cdot \mc S_\cdot \mc F| } \]
where $f$ the inclusion of the object $(S^0)$ in $w_0 \mc F \rightarrow w_0 \mc S_1 \mc F$.
Note that $\overline{f}: \Sph \rightarrow B^\infty \Omega|w_\cdot \mc S_\cdot \mc F|$ is an equivalence of spectra, using Thm \ref{alternate_delooping} to identify the latter spectrum with $K(\mc F) \simeq \Sph$.

Now apply Proposition \ref{transfer_via_operad} with $X = \Omega|w_\cdot \mc S_\cdot \mc F| \simeq QS^0$ and $Y = S^0$.
We conclude that this is the transfer and collapse map of $BG \times BG$ along the bundle whose fiber is $G$ and whose $G \times G$-monodromy is given by left and right multiplication.
By Prop \ref{norm_of_trivial_action}, this is the equivariant norm map of $\Sph$.

To move from finite sets to modules over a ring spectrum $R$, we observe that the assembly and coassembly maps for $K(R[G])$ fit into the bottom row of a diagram
\[ \xymatrix @C=0.45in{
BG_+ \sma \Sph \sma K(R) \ar[d] \\
BG_+ \sma A(*) \sma K(R) \ar[r]^-{\textup{assem} \sma \id} \ar[d]^-{\eta} & A(BG) \sma K(R) \ar[r]^-{\textup{Cartan} \sma \id} \ar[d]^-{\eta} & \uda(BG) \sma K(R) \ar[r]^-{\textup{coassem} \sma \id} \ar[d]^-{\eta} & F(BG_+,A(*) \sma K(R)) \ar[d]^-{\eta} \\
BG_+ \sma K(R) \ar[r]^-{\textup{assem}} & A(BG;R) \ar[r]^-{\textup{Cartan}} & \uda(BG;R) \ar[r]^-{\textup{coassem}} & F(BG_+,K(R)) } \]
in which the pairings $\eta$ are described at the end of section \ref{pairing_section}.
By Propositions \ref{pairings} and \ref{obvious}, this diagram commutes up to homotopy.
The composite of the vertical maps on the left-hand side is an equivalence, so the composite of assembly and coassembly can be evaluated by determining the image of $BG_+ \sma K(R)$ along the top part of the diagram.
The desired map has adjoint
\[ (BG \times BG)_+ \sma K(R) \ra K(R) \]
which is the smash product of the identity map of $K(R)$ and the map we just examined above.
Applying Prop \ref{norm_of_trivial_action} again, we conclude that this is the adjoint of the equivariant norm of $K(R)$.
\end{proof}

\begin{rmk}
Our key claim about the $K$-theory of finite sets also follows from a $THH$ result in the author's thesis (\cite{malkiewich2014duality}, section 3.7).
This reduction is possible because the assembly and coassembly maps on $K$-theory commute with the trace maps into $THH$, and the composite map $K(\mc F) \ra THH(\mc M_f(1;\Sph))$ is an equivalence.
We will save the $THH$-level argument for a future paper, since it seems to generalize well but uses somewhat elaborate geometric coherence machinery.
\end{rmk}

\section{A generalization to all subgroups.}\label{all_subgroups}

In this final section, we briefly examine a broader collection of assembly maps, and relate them to the Segal conjecture.

We continue to assume that $G$ is a finite group. If $\Sph$ denotes a fibrant version of the genuinely $G$-equivariant sphere spectrum, then we may consider its tom Dieck splitting and the map from its fixed points to homotopy fixed points:
\[ \bigvee_{(H) \leq G} \Sigma^\infty_+ BWH \simar \Sph^G \ra \Sph^{hG} = F(BG_+, \Sph) \]
Here $(H)$ denotes conjugacy classes of subgroups, and $WH = NH/H$ is the Weyl group of $H$.
The first map above is always an equivalence, by tom Dieck's splitting theorem.
If $G$ is a $p$-group, then the second map is an equivalence after $p$-completion, by the affirmed Segal conjecture \cite{carlsson1984equivariant}.
More generally, it is an equivalence after a certain kind of completion at the augmentation ideal of the Burnside ring.

We recall a nonequivariant description of this composite from \cite{lewis1982classifying}, which one can also verify using the method of Prop \ref{nonequivariant_norm}.
\begin{prop}\label{what_segal_map_is}
The above map $\Sigma^\infty_+ BWH \rightarrow F(BG_+,\Sph)$ is adjoint to a transfer along the bundle
\[ EG \times_{NH} EWH^\ell \ra BG \times BWH \]
followed by the collapse of $EG \times_{NH} EWH^\ell$ to a point.
\end{prop}

It is not hard to check that this bundle has fiber
\[ G \times_{NH} WH \cong G/H \]
with $G \times WH$-monodromy given by left and right multiplication:
\[ (g,n) \cdot (\gamma H) = g\gamma H n^{-1} = g\gamma n^{-1}H, \qquad g,\gamma \in G, \quad n \in NH \]

Now consider the composite
\[ \xymatrix @C=0.6in{ BWH_+ \sma K(R) \ar[r]^-{\textup{assembly}} & K(R[WH]) \ar[r]^-{G_+ \sma_{NH} -} & G^R(R[G]) \ar[r]^-{\textup{coassembly}} & F(BG_+,K(R)) } \]
The exact functor in the middle takes the $WH$-module $M$ to the $G$-module $G_+ \sma_{NH} M$.
In particular, this sends the module $R \sma WH_+$ to the module $R \sma G/H_+$.

\begin{thm}\label{assembly_coassembly_all_subgroups}
This composite is adjoint to the identity map of $K(R)$ smashed with the map of the Segal conjecture from Prop \ref{what_segal_map_is}.
\end{thm}

\begin{proof}
As in the proof of Theorem \ref{assembly_coassembly_norm}, it suffices to examine the effect of these maps on the $K$-theory of finite sets. We let $n = |G/H|$ and fix a bijection between $G/H$ and the standard set of $n$ elements.
Define a homomorphism $\Gamma = G \times WH \rightarrow \Sigma_n$ by sending the pair $(g,n)$ to the permutation $x \mapsto gxn^{-1}$.
The composite of the assembly and coassembly maps on finite sets becomes
\[ (BG \times BWH)_+ \ra K(\mc F) \simeq \Omega^\infty S^\infty \]
defined simplicially by
\[ N_k G \times N_k WH \ra w_k \mc F \]
\[ (g_1,\ldots,g_k;n_1 H,\ldots,n_k H) \mapsto \xymatrix @C=0.5in{ G/H_+ \ar[r]^-{g_1^{-1} \cdot - \cdot n_1} & G/H_+ \ar[r]^-{g_2^{-1} \cdot - \cdot n_2} & \ldots \ar[r]^-{g_k^{-1} \cdot - \cdot n_k} & G/H_+ } \]
and then we apply the usual inclusion $|w_\cdot \mc F| \rightarrow \Omega|w_\cdot \mc S_\cdot \mc F|$.
For the same reasons as before, this must be the transfer and collapse map of $BG \times BWH$ along the bundle whose fiber is $G/H$ and whose $G \times WH$-monodromy is given by left and right multiplication.
By Prop \ref{what_segal_map_is}, this is the map of the Segal conjecture.
\end{proof}

We have included this generalization, because it allows us to partially compute the Swan theory of $\Sph[G]$, when $G$ is a finite $p$-group.
This was stated in the introduction under Theorem \ref{intro_burnside_split}.

\begin{cor}
If $R$ is a ring spectrum augmented over the sphere, and $G$ is a finite $p$-group, then the group $\pi_0(G^R(R[G])^\wedge_p)$ contains the Burnside ring $A(G)^\wedge_p$ as a direct summand.
\end{cor}

\begin{proof}
Since $R$ is augmented, so is $K(R)$.
We choose $\eta: \Sph \rightarrow K(R)$ by picking out $R$ as a module over itself.
We define $\epsilon: K(R) \rightarrow \Sph$ by composing the topological Dennis trace $K(R) \rightarrow THH(R)$ with the augmentation $THH(R) \rightarrow THH(\Sph) \simeq \Sph$.
Though the Dennis trace is difficult to describe on most elements, it is easy to check that the sphere in $K(R)$ picked out by $\eta$ goes to the unit map $\Sph \ra R$ at level 0 of the cyclic bar construction for $THH(R)$.
So $\epsilon \circ \eta$ can be identified up to homotopy with the map induced on $THH$ by the map of rings $\Sph \rightarrow R \rightarrow \Sph$, but this is just the identity map of $\Sph$.

Now take the maps of Thm \ref{assembly_coassembly_all_subgroups}, for varying $H$, and compose them with the augmentation of $K(R)$:
\[ \xymatrix{
\bigvee_{(H) \leq G} \Sigma^\infty_+ BWH \ar[d]^-{\id \sma \eta} & \\
\bigvee_{(H) \leq G} \Sigma^\infty_+ BWH \sma K(R) \ar[r]^-\alpha & \bigvee_{(H) \leq G} K(R[WH]) \ar[r]^-{G_+ \sma_{NH}} & G^R(R[G]) \ar[r]^-{c\alpha} & F(BG_+, K(R)) \ar[d]^-{F(\id,\epsilon)} \\
&&& F(BG_+, \Sph) } \]
This is adjoint to the maps of the Segal conjecture, smashed with the composite $\epsilon \circ \eta \simeq \id$.
So the above composite agrees with the map of the Segal conjecture.

Now let $G$ be a finite $p$-group, take the $p$-completion of every spectrum above, and take $\pi_0$ to get a diagram of abelian groups.
Then the first and last terms are isomorphic to the $p$-completed Burnside ring $A(G)^\wedge_p$, and the map between them is an isomorphism.
It follows that every group along the middle row has $A(G)^\wedge_p$ as a direct summand.
\end{proof}

\bibliographystyle{amsalpha}
\bibliography{coassembly}{}

Department of Mathematics \\
University of Illinois at Urbana-Champaign \\
1409 W Green St \\
Urbana, IL 61801 \\
\texttt{cmalkiew@illinois.edu}

\end{document}